\documentclass[12pt,twoside]{amsart}
\usepackage{amssymb}
\usepackage{amscd}
\usepackage[abbrev,alphabetic]{amsrefs}
\usepackage{hyperref}
\usepackage{comment}
\usepackage{array,multirow,tabularx,longtable}
\usepackage{tikz}
\usetikzlibrary{cd}
\usepackage{here}
\usepackage{multirow}
\usepackage[margin=1.25in]{geometry}
\usepackage{mathtools}

\usepackage{framed}
\usepackage{fancybox}
\usepackage{ascmac}

\title[Mukai bundles on prime Fano threefolds]
{Mukai bundles and Brill-Noether generality for prime Fano threefolds in positive characteristic} 

\author{Hiromu Tanaka} 
\address{Department of Mathematics, 
Graduate School of Science, 
Kyoto University, 
Kyoto 606-8502, JAPAN} 
\email{tanaka.hiromu.7z@kyoto-u.ac.jp}

\newcommand{\cHom}{\mathcal{H}om}
\newcommand{\cExt}{\mathcal{E}xt}

\DeclareMathOperator{\Ev}{ev}

\newcommand{\ol}{\overline}

\newcommand{\wt}{\widetilde}

\newcommand{\Br}[0]{{\operatorname{Br}}}

\newcommand{\rank}[0]{\operatorname{rank}}

\newcommand{\Ker}[0]{\operatorname{Ker}}

\renewcommand{\Im}[0]{\operatorname{Im}}

\newcommand{\Spec}[0]{\operatorname{Spec}}

\newcommand{\Hom}[0]{{\operatorname{Hom}}}
\newcommand{\Bs}[0]{\operatorname{Bs}}
\newcommand{\Supp}[0]{\operatorname{Supp}}
\newcommand{\Pic}[0]{\operatorname{Pic}}

\newcommand{\Ex}[0]{{\operatorname{Ex}}}

\newcommand{\Frac}[0]{\operatorname{Frac}}

\newcommand{\Gr}[0]{{\operatorname{Gr}}}

\newcommand{\Ext}[0]{{\operatorname{Ext}}}

\newcommand{\Gal}[0]{{\operatorname{Gal}}}

\newtheorem{thm}{Theorem}[section]
\newtheorem{lem}[thm]{Lemma}
\newtheorem*{lem*}{Lemma}

\newtheorem{cor}[thm]{Corollary}

\newtheorem{prop}[thm]{Proposition}

\newtheorem*{claim*}{Claim}         
\newtheorem{step}{Step}

\theoremstyle{definition}

\newtheorem{dfn}[thm]{Definition}

\newtheorem{rem}[thm]{Remark}

\newtheorem{nota}[thm]{Notation}         

\makeatletter
  
  \@addtoreset{equation}{thm}
  \makeatother

\newcommand{\bE}{\mathbb{E}}

\newcommand{\cE}{\mathcal{E}}

\newcommand{\cH}{\mathcal{H}}

\newcommand{\cM}{\mathcal{M}}

\newcommand{\cO}{\mathcal{O}}
\newcommand{\cP}{\mathcal{P}}
\newcommand{\cQ}{\mathcal{Q}}

\newcommand{\cS}{\mathcal{S}}

\newcommand{\cU}{\mathcal{U}}
\newcommand{\cV}{\mathcal{V}}

\newcommand{\cX}{\mathcal{X}}
\newcommand{\cY}{\mathcal{Y}}
\newcommand{\cZ}{\mathcal{Z}}

\newcommand{\U}{\mathcal{U}}

\newcommand{\MO}{\mathcal{O}}

\newcommand{\Q}{\mathbb{Q}}
\newcommand{\Z}{\mathbb{Z}}

\renewcommand{\P}{\mathbb{P}}

\newcommand{\m}{\mathfrak{m}}

\usepackage{listings}
\begin{document}

\maketitle

\begin{abstract}
Let $X$ be a prime Fano threefold of genus $g\geq 8$ in positive characteristic. 
For every nontrivial decomposition $g =r s$, we prove that there exists a unique Mukai bundle on $X$ of type $(r, s)$. 
To this end, we show that 
every smooth hyperplane section of $X$ is Brill-Noether general. 
\end{abstract}


\tableofcontents

\section{Introduction}

The classification of Fano varieties is one of the central themes in algebraic geometry.
In dimension one, the only Fano variety is the projective line $\P^1$.
Fano varieties of dimension two are known as del Pezzo surfaces, and their classification is classical: every del Pezzo surface is isomorphic either to $\P^1 \times \P^1$ or to the blowup of $\P^2$ at no more than eight points in general position.

The study of Fano varieties in dimension three goes back to Fano.
Over an algebraically closed field of characteristic zero, the classification of Fano threefolds was completed by Mori and Mukai
\cite{MM81}, \cite{MM83}, \cite{MM03},
building on earlier work of Iskovskih and Shokurov
\cite{Isk77}, \cite{Isk78}, \cite{Sho79a}, \cite{Sho79b}; 
see also \cite{IP99}, \cite{Tak89}.
Subsequently, Mukai announced explicit geometric descriptions of prime Fano threefolds
\cite{Muk89}.
These descriptions were later established by Bayer--Kuznetsov--Macri
\cite{BKM24}, \cite{BKM25}.

\begin{dfn} Let $k$ be an algebraically closed field. 
We say that $X$ is a {\em prime Fano threefold} over $k$ if $X$ is a smooth projective threefold over $k$ 
such that $-K_X$ is ample and $\Pic X$ is generated by $\omega_X$. The positive integer $g$ defined by $(-K_X)^3 = 2g-2$ is called the {\em genus} of $X$. 
\end{dfn}

In positive characteristic, the Mori--Mukai classification has been established by Asai and the author 
\cite{FanoI}, \cite{FanoII}, \cite{FanoIII}, \cite{FanoIV}.
This naturally raises the problem of obtaining a positive-characteristic counterpart of Mukai's description of prime Fano threefolds.
The aim of the present paper, together with its sequel 
\cite{KT-primitive}, is to establish such a description in positive characteristic.
The main result of this paper is the existence of Mukai bundles on prime Fano threefolds, which provides the key step toward establishing Mukai's description in positive characteristic.


\begin{thm}[Theorem \ref{t Mukai bdl X}, 
Proposition \ref{p Mukai bdl stable}, Proposition \ref{p uniqueness Mukai bdl}]\label{intro Mukai bdl}
Let $k$ be an algebraically closed field of characteristic $p>0$. 
Let $X$ be a prime Fano threefold over $k$ 
of genus $g \geq 8$. 
Take integers $r$ and $s$ satisfying $r \geq 2$, $s \geq 2$, 
and $g =rs$. 
Then there exists a $\mu_{-K_X}$-stable vector bundle $\cU_X$ on $X$ such that 
\begin{enumerate}
\item 
$\rank(\cU_X) = r$, $c_1(\cU_X) = K_X$, 
\item 
$H^j(X,\cU_X) = 0$ for every $j \geq 0$, 
\item 
$\cU_X^\vee := \cHom_{\MO_X}(\cU_X, \MO_X)$ is globally generated, and 
\item 
$\dim H^0(X, \cU_X^\vee)=r+s$, $H^i(X, \cU_X^\vee) = 0$ for every $i>0$.
\end{enumerate}
Furthermore, such a vector bundle is unique up to isomorphisms. 
\end{thm}

To prove the existence of Mukai bundles, we establish the following Brill-Noether generality result: 

\begin{thm}[Theorem \ref{t main2}]\label{intro BN}
Let $k$ be an algebraically closed field of characteristic $p>0$. 
Let $X$ be a prime  Fano threefold over $k$ 
such that $|-K_X|$ is very ample. 
Then the following hold: 
\begin{enumerate}
\item 
$(S, H)$ is Brill-Noether general for a smooth member $S$ of $|-K_X|$ and $H := -K_X|_S$. 
\item If $S_1$ and $S_2$ are smooth members of $|-K_X|$ 
such that $C := S_1 \cap S_2$ is a smooth curve, then $C$ is Brill-Noether general. 
\end{enumerate}
\end{thm}


\subsection{Proof of Theorem \ref{intro Mukai bdl}}

Let $X$ be a prime Fano threefold over $k$ 
of genus $g  =rs \geq 8$ satisfying $r \geq 2$ and $s \geq 2$. 
The existence problem of a vector bundle $\cU_X$ as in Theorem \ref{intro Mukai bdl} is reduced to the case when $s \geq r$ (Lemma \ref{l perp is Mukai}). 
Then 
it is enough to find a pair $(S, M)$ such that 
\begin{itemize}
\item $S$ is a member of $|-K_X|$ which is a  canonical (i.e., RDP) K3 surface. 
\item $M$ is a special Mukai divisor of type $(r, s)$ 
on $(S, H)$ for $H:=-K_X|_S$, i.e., 
$M$ is a Cartier divisor on $S$ such that 
\begin{enumerate}
\renewcommand{\labelenumi}{(\roman{enumi})}
\item $M \cdot H = (r-1)(s+1)$, 
\item $h^0(\MO_S(M))=r$, 
\item $H^1(\MO_S(M))= H^2(\MO_S(M))=0$, and 
\item $|M|$ is base point free. 
\end{enumerate}
\end{itemize}
Indeed, we can check that the vector bundle $\cU_X$  defined by the exact sequence 
\[
0 \to \cU_X \to H^0(M) \otimes \MO_X \to \MO_S(M) \to 0 
\]
satisfies the required properties. 
In characteristic zero, Bayer--Kuznetsov--Macri proved the existence of such a pair up to taking the minimal resolution of $S$ \cite[Proposition 5.5]{BKM24}. 
However, 
it seems to be hard to apply the same argument to the case of positive characteristic. 

Our strategy is to utilise a $W(k)$-lift $\cX$ of $X$, 
whose existence is guaranteed by \cite[Theorem A]{KTLift1}. 
By taking a $W(k)$-lift of a very general Lefshetz pencil of $X \subset \P^{g+1}_k$, 
we can find a finite extension $R$ of $W(k)$ and a mixed-characteristic family 
$\pi : \cS \to \Spec R$ of canonical K3 surfaces 
such that the generic fibre $S_Q$ with $Q := \Frac R$ admits a 
Cartier divisor $M_Q$ whose base change $M_{\ol{Q}}$ to the algebraic closure $\ol{Q}$ is a special Mukai divisor $M_Q$  of type $(r, s)$.  

We now treat the case when 
\begin{enumerate}
\item[($\star$)] $\cS$ is factorial (e.g., $\cS$ is regular). 
\end{enumerate}
In this case, we can find the specialisation $M$ of $M_Q$,  
because the Weil divisor $\cM$ on $\cS$, obtained as the Zariski closure of $M_Q$, 
is automatically Cartier by $(\star)$. 
Then it suffices to check that 
the Cartier divisor $M$ on the positive-characteristic fibre $S$ of $\pi$ satisfies (i)-(iv). 
The condition (i) follows from the invariance of intersection numbers $M \cdot H = M_Q \cdot H_Q = (r-1)(s+1)$, 
where $H := -K_{\cX}|_S$ and $H_Q := -K_{\cX}|_{S_Q}$. 
The condition (ii) holds by 
\[
g =rs = 
h^0(\MO_{S_Q}(M_Q))h^0(\MO_{S_Q}(H_Q-M_Q)) \overset{{\rm (a)}}{\leq} 
h^0(\MO_S(M))h^0(\MO_S(H-M)) 
\overset{{\rm (b)}}{\leq} 
g, 
\]
where (a) follows from the upper semi-continuity and (b) is a conclusion of the Brill-Noether generality (Theorem \ref{intro BN}(1)). 
By standard argument, (ii) implies (iii). 
The condition (iv) is automatic by Remark \ref{intro rem min resol}. 
This completes the case when $(\star)$ holds.

\begin{rem}\label{intro rem min resol}
Let $\mu : \wt{S} \to S$ be the minimal resolution of $S$ and set $\wt{H} := \mu^*H$. 
In \cite{BKM24}, 
they prove that 
the minimal resolution 
$(\wt{S}, \wt{H})$ admits a special Mukai divisor of type $(r, s)$. 
In our present situation, we establish the following results which enable us to ignore the condition (iv) 
and the difference between $S$ and $\wt{S}$:  
\begin{enumerate}
\item If $M$ satisfies (i)-(iii), then the condition (iv) holds automatically (Proposition \ref{p weakly is not weakly}). 
\item $(S, H)$ has a special Mukai divisor of type $(r, s)$ if and only if $(\wt{S}, \wt{H})$ has a special Mukai divisor of type $(r, s)$ (Corollary \ref{c exist Mukai div up to resol}).  
\end{enumerate}
\end{rem}

% ...

Although 
it is too much to hope $(\star)$ in the general case, 
we can apply the above argument after replacing $\cS$ by a suitable birational model $\cS'$. 
Unfortunately, Artin's simultaneous resolution is not 
enough for our purpose. 
Instead, we will apply minimal model program. 
More specifically, 
by using a desingularisation and minimal model program, 
we can find a projective birational morphism $\cS' \to \cS$ such that 
$\cS'$ is a $\Q$-factorial terminal threefold and 
the induced morphism $\cS'_t \to \cS_t$ for each geometric point $t$ of $\Spec R$ is  either an isomorphism or the minimal resolution of $\cS_t$. 
Note that we may ignore the difference between the fibres 
$\cS_t$ and $\cS'_t$ by Remark \ref{intro rem min resol}. 
Since $\cS'$ has only isolated hypersurface singularities, 
every $\Q$-Cartier Weil divisor is Cartier 
\cite[Theorem 1.1.3]{CS24}. 
This, together with $\Q$-factoriality, implies that $\cS'$ is factorial. 


\begin{rem}
In contrast to the situation in characteristic zero, 
normality of the plt centre $\cS'_p$ of the resulting plt pair $(\cS', \cS'_p)$ is not automatic.  
In the present situation, 
normality of $\cS'_p$ follows from a kind of inversion of adjunction via BCM regularity established by 
\cite{MS21}, \cite{MSTWW}. 
\end{rem}



\subsection{Proof of Theorem \ref{intro BN}}

Let $X$ be a prime Fano threefold 
such that $|-K_X|$ is very ample. 
Take two smooth members $S_1$ and $S_2$ of $|-K_X|$ 
such that $C:=S_1 \cap S_2$ is a smooth curve. 
It is enough to show that $C$ is Brill-Noether general. 
Pick a Cartier divisor $A$ on $C$ such that $|A|$ and $|K_C -A|$ are base point free. 
It suffices  to prove $h^0(A) h^1(A) \leq g$. 

The main idea is to use the generic member $T_0$ of the pencil generated by $S_1$ and $S_2$. 
Note that $T_0$ is a smooth prime divisor on 
the base change $X \times_k \kappa_0$ for $\kappa_0 := k(t)$. 
In order to avoid inseparable phenomena, we further take 
the base changes $T := T_0 \times_{\kappa_0} \kappa$ and  $X_{\kappa} := X \times_k \kappa$ to the purely inseparable closure 
$\kappa :=\kappa_0^{1/p^{\infty}} = k(t^{1/p^{\infty}})$. 
Since the Lefshcetz hyperplane section theorem holds for the generic member $T_0$ \cite{Tan24}, 
we have $\Pic T_0 = \Z (K_X|_{T_0})$. 
This, together with primitivity of $-K_{X_{\ol{\kappa}}}|_{T_{\ol{\kappa}}}$ (Theorem \ref{t primitive -K_X|_S}), 
implies that $\Pic T = \Z H$ for $H := -K_X|_T$. 
For the elementary transform $F_{T, C, A}$  defined by 
\[
0 \to F_{T, C, A} \to H^0(A) \otimes_k \MO_T \to \MO_C(A) \to 0, 
\]
$F := F_{T, C, A}$ is a $\mu$-stable vector bundle.

If the base change $F_{\ol{\kappa}}$ to the algebraic closure $\ol{\kappa}$ is still $\mu$-stable, then it is easy to show 
the required inequality $h^0(A) h^1(A) \leq g$ 
by using $\dim \Hom(F_{\ol{\kappa}}, F_{\ol{\kappa}}) =1$. 
Since this is too much to expect, we instead consider the 
endomorphism ring $\Hom(F, F)$ before taking the base change. 
As $F$ is $\mu$-stable, 
every nonzero endomorphism is an automorphism. 
In other words, $\Hom(F, F)$ is a division ring. 
As the current base field $\kappa = k(t^{1/p^{\infty}})$ is a perfect $C_1$-field, $\Hom(F, F)$ is a commutative field which is a finite separable extension of $\kappa$. 
Taking the base change to the algebraic closure $\ol{\kappa}$, 
we obtain ring isomorphisms 
\[
\Hom(F_{\ol{\kappa}}, F_{\ol{\kappa}}) 
\simeq 
\Hom(F, F) \otimes_{\kappa} \ol{\kappa} \simeq \ol{\kappa}_1 \times \cdots\times \ol{\kappa}_n 
\]
for $F_{\ol{\kappa}} := F \otimes_\kappa \ol{\kappa}$ and $\ol{\kappa}_i := \ol{\kappa}$. 
Corresponding to these isomorphisms, 
we get a direct sum decomposition into vector bundles
\[
F_{\ol{\kappa}} = F_1 \oplus \cdots \oplus F_n. 
\]
As the direct product factors $\ol{\kappa}_1, ..., \ol{\kappa}_n$ are $\Gal(\ol{\kappa}/\kappa)$-conjugate, 
so are $F_1, ..., F_n$. 
Although $F_1, ..., F_n$ are not necessarily $\ol{\kappa}$-isomorphic, they are $\kappa$-isomorphic. 
Hence they share certain invariants, e.g., 
$\rank F_1= \cdots = \rank F_n$. 
Moreover, we get 
\[
H = -c_1(F) = -c_1(F_1)- \cdots -c_1(F_n) = H_1 + \cdots +H_n
\]
for $H_i := -c_1(F_i)$. 
As $F^\vee$ is globally generated, so are $F_i^\vee$ and 
its determinant bundle $\det F_i^\vee = \MO_T(H_i)$. 
Since $H_1^2 = \cdots = H_n^2$, we get 
\begin{equation}\label{e intro BN}
    13 \geq g+1 \overset{(\star)}{\geq} h^0(H_1) h^0(H_2 + \cdots +H_n) 
\geq h^0(H_1) h^0(H_2) \geq 
\left( \frac{1}{2}H_1^2 + 2\right)^2, 
\end{equation}
where $(\star)$ is a weakening of the Brill-Noether inequality (Proposition \ref{p weak BN generality}), which is established again by using the elementary transform: 
\[
0 \to F_{X_{\kappa}, T, H_1} \to H^0(H_1) \otimes \MO_{X_\kappa} \to \MO_T(H_1) \to 0. 
\]
As a conclusion of (\ref{e intro BN}), 
we get  $H_1^2 <4$. 
The remaining cases $H^2_1 =0$ and $H_1^2 =2$ are handled by a similar but more careful analysis. 
For further details on the overall argument, see Section \ref{s BN general}.






\medskip
\noindent {\bf Acknowledgements.}
The author was supported by JSPS KAKENHI Grant number JP22H01112 and JP23K03028. 
The author thanks Akihiro Kanemitsu for fruitful discussions and helpful suggestions that led to more general statements and simpler proofs of Lemma \ref{l stablity criterion} and Theorem \ref{t primitive -K_X|_S}. 
ChatGPT (OpenAI) was used during the preparation of this article for assistance with mathematical exploration, computations, and language editing.

\section{Preliminaries}

\subsection{Notation}\label{ss notation}

In this subsection, we summarise notation used in this paper. 

\begin{enumerate}
\item We will freely use the notation and terminology in \cite{Har77}. 
In particular, $D_1 \sim D_2$ means linear equivalence of Weil divisors. 
\item 
Throughout this paper, 
we work over an algebraically closed field $k$ 
of characteristic $p>0$ unless otherwise specified. 
\item For an integral scheme $X$, 
we define the {\em function field} $K(X)$ of $X$ 
as the local ring $\MO_{X, \xi}$ at the generic point $\xi$ of $X$. 
For an integral domain $A$, $K(A)$ denotes the function field of $\Spec A$. 
\item 
Our notation will not distinguish between invertible sheaves and 
Cartier divisors. For example, $\Pic X= \Z K_X$ 
means that $\Pic X$ is generated by $\omega_X$ as an abelian group. 
\item We say that $X$ is a {\em variety} over a field $\kappa$ if 
$X$ is a separated integral scheme which is of finite type over $\kappa$ such that $\kappa$ is algebraically closed in $K(X)$. 
When $X$ is a projective integral normal scheme over $\kappa$, 
$\kappa$ is algebraically closed in $K(X)$ if and only if $H^0(X, \MO_X)=\kappa$. 
We say that $X$ is a {\em curve} (resp. a {\em surface}, resp. a {\em threefold}) over $\kappa$  
if $X$ is a variety over $\kappa$ of dimension one (resp. {\em two}, resp. {\em three}). 
\item 
For a scheme $X$ over a field $\kappa$ and a field extension $\kappa'/\kappa$, 
we set $X \times_\kappa \kappa' := X \times_{\Spec \kappa} \Spec \kappa'$. 
\item\label{def exc locus mathbbE} 
Given  a projective normal surface $S$ and 
a nef and big Cartier divisor $H$ on $S$, 
the {\em exceptional locus} 
$\bE(H)$ of $H$ is defined by $\bE(H) := \bigcup_{C, H \cdot C=0} C$, where $C$ runs over all the curves on $S$. 
It is well known that $\bE(H)$ is a proper closed subset of $S$.
\item 
Given a projective scheme $X$ over a field with $\rho(X) =1$ and a vector bundle $E$, 
we say that $E$ is {\em $\mu$-stable} 
if $E$ is $\mu_H$-stable for some (every) ample Cartier divisor $H$. 
\end{enumerate}

\subsubsection{Fano threefolds}

We say that $X$ is a {\em Fano threefold} (over $k$) 
if $X$ is a smooth projective threefold over $k$ such that $-K_X$ is ample. 
We say that $X$ is a {\em prime Fano threefold} 
if $X$ is a Fano threefold such that $\Pic X  = \Z K_X$. 
For a prime Fano threefold $X$, 
the integer $g$ defined by $(-K_X)^3 = 2g -2$ is called the {\em genus} of $X$. 
It is well known that the following hold  \cite[Theorem 1.1]{FanoI}, \cite[Theorem 1.2]{FanoII}: 
\begin{enumerate}
\item[(i)] $2 \leq g \leq 12$ and $g \neq 11$. 
\item[(ii)] 
If $|-K_X|$ is very ample, then $g \geq 3$. 
If $g \geq 4$, then $|-K_X|$ is very ample. 
\end{enumerate}
If $X$ is a prime Fano threefold and $|-K_X|$ is very ample, 
then we have the closed embedding $X \subset \P^{g+1}$ induced by $|-K_X|$. 
In this case, $X \subset \P^{g+1}$ is also called a prime Fano threefold of genus $g$.

\subsubsection{Mukai vectors of vector bundles on K3 surfaces}

Let $S$ be a smooth K3 surface and let $E$ be a vector bundle $E$ on $S$. 
Set 
\[
s(E) := \rank E + {\rm ch}_2(E) = \rank E + \frac{1}{2} c_1(E)^2-c_2(E)
\]
and the {\em Mukai vector} $v(E)$ of $E$ is defined by 
\[
v(E) := (\rank E, c_1(E), s(E)) \in \Z \oplus \Pic(S) \oplus \Z. 
\]
The Riemann-Roch theorem implies 
\[
\chi(S, E) = (\rank E) \cdot \chi(S, \MO_S) + 
\frac{1}{2} c_1(E) \cdot (c_1(E) -K_S) -c_2(E) = 
\rank E + s(E). 
\]
Moreover, 
we have $v(E) = v(E') + v(E'')$ for an exact sequence $0 \to E' \to E \to E'' \to 0$ of vector bundles.

\subsection{Elementary transform}

\begin{dfn}\label{d ele tf}
Let $X$ be a projective variety over a field $\kappa$. 
For a nonzero effective Cartier divisor $T$ and a globally generated 
line bundle  $L$ on $T$, 
we define the coherent sheaf $F_{X, T, L}$ on $X$ by the following exact sequence: 
\begin{equation}\label{e1 ele tf}
0 \to F_{X, T, L} \to H^0(T, L) \otimes_\kappa \MO_X \xrightarrow{r} i_*L \to 0,
\end{equation}
where $i: T \hookrightarrow X$ denotes the induced closed immersion. 
Here  $r$ denotes the restriction homomorphism, 
which is surjective as $L$ is globally generated. 
\end{dfn}

\begin{prop}\label{p ele tf}
We use the same notation as in Definition \ref{d ele tf}. 
Then the following hold. 
\begin{enumerate}
\item $F_{X, T, L}$ is a vector bundle on $X$ of rank 
$h^0(T, L)$. 
\item $c_1(F_{X, T, L}) = -T$. 
\item $c_2(F_{X, T, L}) = i_*c_1(L)$. 
\item $H^0(X, F_{X, T, L})=0$ and 
$H^1(X, F_{X, T, L}) \hookrightarrow H^0(T, L) \otimes_\kappa H^1(X, \MO_X)$. 
\item $F^\vee_{X, T, L}$ is globally generated at the generic point if $X$ is regular. 
\item 
$F_{X, T, L}^\vee$ is globally generated 
if $X$ is regular, $H^1(X, \MO_X)=0$, and 
$\MO_X(T)|_T \otimes L^{-1}$ is globally generated. 
\end{enumerate}
\end{prop}

\begin{proof}
The assertions (1)--(3) follow from the same argument as in \cite[Lemma 16 in Section 2]{Fri98}. 
The assertion (4) holds by taking  the cohomologies $H^i(X, -)$ of the exact sequence (\ref{e1 ele tf}). 

Let us show (5) and (6). 
By applying $\cHom(-, \MO_X)$ to  (\ref{e1 ele tf}), 
we get 
\[
0 \to H^0(T, L)^\vee \otimes \MO_X \to F^\vee_{X, T, L} 
\to \MO_X(T) \otimes i_*L^{-1} \to 0. 
\]
Hence (6) holds. 
Moreover, (5) holds because  
$H^0(T, L)^\vee \otimes \MO_X$ is globally generated and 
$(H^0(T, L)^\vee \otimes \MO_X)_{\xi} \xrightarrow{\simeq} 
(F^\vee_{X, T, L})_{\xi}$ for the generic point $\xi$. 
\qedhere



\end{proof}

\begin{lem}\label{l DPS bundle}
Let $X$ be a regular variety over a field $\kappa$. 
Let $E$ be a vector bundle and let $F$ be a reflexive subsheaf of $E$. 
Assume that the cokernel of $(\wedge^m F)^{**} \to \wedge^m E$ is locally free 
for $m:= \rank F$, where $(\wedge^m F)^{**}$ denotes the double dual 
of $\wedge^m F$. 
Then $F$ and $E/F$ are locally free. 
\end{lem}

\begin{proof}
The same argument as in \cite[Lemma 1.20]{DPS94} works. 
\end{proof}

\begin{lem}\label{l trivial criterion}
Let $X$ be a regular projective variety over a field $\kappa$ and 
let $F$ be a reflexive sheaf on $X$ of rank $m$. 
Assume that $H^0(X, (\wedge^m F)^{**} ) \neq 0$ and 
there exists  an injective $\MO_X$-module homomorphism 
$\iota : F \hookrightarrow \MO_X^{\oplus n}$ for some $n>0$. 
Then $F \simeq \MO_X^{\oplus m}$. 
\end{lem}

\begin{proof}
We now show that $F$ is locally free. 
Since $X$ is regular, $\det F:=(\wedge^m F)^{**}$ is an invertible sheaf. 
The injection 
$\iota:F\hookrightarrow \MO_X^{\oplus n}$ 
induces an injective $\MO_X$-module homomorphism
\[
\iota': 
\det F
\hookrightarrow 
\wedge^m \MO_X^{\oplus n}
\simeq
\MO_X^{\oplus N},
\qquad
N:=\binom{n}{m}.
\]
Taking the projection onto a suitable direct summand, 
we can find an injection $\det F \hookrightarrow \MO_X$. 
This, together with $H^0(X, \det F ) =H^0(X,  (\wedge^m F)^{**} ) \neq 0$, implies that $\det F \simeq \MO_X$. 
Then $\iota'$ is a split injection, and hence $F$ is locally free by Lemma \ref{l DPS bundle}.

\medskip 

Let $\xi$ be the generic point of $X$. 
If $m<n$, then the image of $\iota_{\xi} : F_{\xi} \hookrightarrow \MO_{X, \xi}^{\oplus n}$ does not contain one of the following elements: 
\[
(1, 0, ..., 0), (0, 1, 0, ..., 0), ..., (0, ..., 0, 1) \in \MO_{X, \xi}^{\oplus n}. 
\]
Hence we can find a coordinate projection $\pi : \MO^{\oplus n}_X \to \MO_X^{\oplus n-1}$ such that the composition $\pi \circ \iota : F \to \MO_X^{\oplus n-1}$ is still injective. 
Repeating this procedure, the problem is reduced to the case when $n=m$. 
Then the same argument as in \cite[the proof of Lemma 3.7(a)]{BKM24} works. 
\end{proof}

\begin{lem}\label{l stablity criterion}
Let $X$ be a regular projective variety over a field $\kappa$ 
such that $\Pic X = \Z H$ for some ample Cartier divisor $H$. 
Let $E$ be a vector bundle on $X$ such that 
\begin{enumerate}
\item $c_1(E) = -H$, 
\item $E^{\vee}$ is globally generated at the generic point, and 
\item 
$H^0(X, E) = 0$.
\end{enumerate}
Then $E$ is $\mu$-stable. 
\end{lem}

\begin{proof}
Suppose that $E$ is not $\mu$-stable. 
Then there exists a nonzero proper reflexive subsheaf $F \subset E$ such that $\rank F < \rank E$ and 
\[
\frac{c_1(F) \cdot H^{d-1}}{\rank F}= \mu_H(F) \geq \mu_H(E) = \frac{c_1(E) \cdot H^{d-1}}{\rank E}. 
\]
By $\Pic X = \Z H$, we have $c_1(F) = sH$ for some $s \in \Z$. 
It holds that 
\[
s H^d = c_1(F) \cdot H^{d-1} \geq 
\frac{\rank F}{\rank E} c_1(E) \cdot H^{d-1}
\overset{{\rm (1)}}{=}  \frac{\rank F}{\rank E} (-H^d) > -H^d, 
\]
which implies $s \geq 0$. 
By (1) and (2), $\MO_X(H)$ is globally generated at the generic point, and hence $H^0(X, \MO_X(H)) \neq 0$. 
In particular, $c_1(F) \geq 0$. 
By (2), we get a generically surjective $\MO_X$-module homomorphism 
$\alpha: \MO_X^{\oplus n} \to E^\vee$ for some $n>0$. 
Applying $(-)^\vee$, 
we get the composition 
$\iota: F \hookrightarrow E \xrightarrow{\alpha^\vee} \MO_X^{\oplus n}$, 
which is generically injective. 
As $F$ is torsion free, $\iota$ is injective. 
Then we obtain  $F \simeq \MO_X^{\oplus m}$ for some $m > 0$ 
(Lemma \ref{l trivial criterion}), which would lead to the following contradiction: 
\[
0 \neq H^0(X, F) \hookrightarrow H^0(X, E) \overset{{\rm (3)}}{=}0. 
\]
Therefore, $E$ is $\mu$-stable. 
\qedhere
\end{proof}

\begin{prop}\label{p ele tf stable}
Let $X$ be a regular projective variety over a field $\kappa$ 
such that $\Pic X = \Z T$ for some prime divisor $T$ on $X$ 
whose complete linear system $|T|$ is base point free. 
Let $L$ be a globally generated line bundle on $T$. 
Then 
$F_{X, T, L}$ is $\mu$-stable. 
\end{prop}

\begin{proof}
It is enough to verify the conditions (1)-(3) of Lemma \ref{l stablity criterion}, which are assured by Proposition \ref{p ele tf}. 
\qedhere

\end{proof}

\subsection{Brill-Noether generality}

\begin{dfn}\label{d BN curve}
Let $C$ be a smooth projective curve of genus $g$. 
We say that $C$ is {\em BN-general} (which stands for Brill-Noether general) if the inequality 
\begin{equation}\label{e1 d BN curve}
h^0(A) h^1(A) \leq g
\end{equation}
holds for every Cartier divisor $A$ on $C$. 
\end{dfn}

\begin{rem}\label{r BN curve deg=1}
Let $C$ be a smooth projective curve of genus $g$. 
For a Cartier divisor $A$ on $C$, the inequality (\ref{e1 d BN curve}) 
is equivalent to each of (1) and (2) below. 
\begin{enumerate}
\item $h^0(A) h^0(K_C -A) \leq g$. 
\item $h^0(A) (h^0(A) +g -1 -\deg A) \leq g$. 
\end{enumerate}
Indeed, (\ref{e1 d BN curve}) is equivalent to (1) (resp.\ (2)) 
by Serre duality (resp.\ the Riemann-Roch theorem). 
In particular,  the inequality (\ref{e1 d BN curve}) automatically holds 
when $h^0(A) \in \{ 0, 1\}$ by (2). 
\end{rem}


\begin{dfn}
\begin{enumerate}
\item 
We say that $S$ is a {\em smooth K3 surface} (resp.\ canonical K3 surface) if 
$S$ is a smooth (resp.\ canonical) projective surface such that $K_S \sim 0$ and $H^1(S, \MO_S)=0$. 
Here a normal surface $S$ is called {\em canonical} if it has at worst canonical singularities. 
\item 
We say that $(S, H)$ is a {\em quasi-polarised smooth K3 surface} 
(resp.\ quasi-polarised canonical K3 surface) if 
$S$ is a smooth (resp.\ canonical) K3 surface and $H$ is a nef and big Cartier divisor on $S$ 
which is linearly equivalent to a smooth prime divisor. 
\end{enumerate}
\end{dfn}

\begin{dfn}
Let $(S, H)$ be a quasi-polarised canonical K3 surface. 
\begin{enumerate}
\item 
The number $g(S, H)$ defined by $H^2 = 2g(S, H) -2$ (i.e., 
$g(S, H) = \frac{1}{2} H^2 +1$) is called the {\em genus} of $(S, H)$. 
\item 
We say that $(S, H)$ is {\em BN-general} if 
the inequality 
\[
h^0(M) h^0(H-M) \leq g(S, H)
\]
holds for every Cartier divisor $M$ on $S$ satisfying $M \not\sim 0$ and $M \not\sim H$. 
\end{enumerate}
\end{dfn}

\begin{rem}
Let $(S, H)$ be a quasi-polarised canonical K3 surface. 
It follows from the Riemann-Roch theorem that $g(S, H) \in \Z$ and $g(S, H) \geq 2$. 
For a smooth member $C \in |H|$, $g(S, H)$ coincides with the genus $g(C)$ of $C$ by the adjunction formula: 
\[
2 g(S, H) -2 = H^2 
=(K_S+H) \cdot H = (K_S+ C) \cdot C = 2g(C) -2. 
\]
\end{rem}

\subsection{Endomorphism rings}

Let $F$ be a coherent sheaf on a projective scheme over a $C_1$-field $\kappa$. 
In this subsection, we summarise some results on the endomorphism rings $\Hom(F, F)$, which are probably well known to experts. 
The base field $\kappa$ is assumed to be a $C_1$-field, 
because our base field $\kappa$ will be, in our application, 
either an algebraically closed field $k$ (the original base field) 
or an algebraic extension of the function field of a curve defined over $k$. 

In what follows, a {\em field} is always assumed to be commutative. 
Although $\Hom(F, F)$ is a (possibly non-commutative) ring, 
its centre $Z(\Hom(F, F))$ is a commutative ring.

\begin{prop}\label{p Artin Wedderburn}
Let $X$ be a projective scheme over a $C_1$-field $\kappa$ and 
let $F$ be a coherent sheaf on $X$ 
such that the endomorphism ring $\Hom(F, F)$ is semisimple. 
Set $n$ to be the number of the maximal ideals of the centre 
$Z(\Hom(F, F))$. 
Then the following hold. 
\begin{enumerate}
\item 
There exist integers $\ell_1, ..., \ell_n >0$ and  
finite field extensions $\kappa_1, ..., \kappa_n$ of $\kappa$
such that the following ring isomorphism holds: 
\begin{equation}\label{e1 Artin Wedderburn}
 \Hom(F, F) \simeq M_{\ell_1}(\kappa_1) \times \cdots \times M_{\ell_n}(\kappa_n).    
\end{equation}
\item 
There exist nonzero coherent subsheaves $F_1, ..., F_n$ of $F$ such that 
\[
F = F_1 \oplus \cdots \oplus F_n, 
\]
$\Hom(F_i, F_j) =0$, and 
the ring isomorphism 
$\Hom(F_i, F_i) \simeq M_{\ell_i}(\kappa_i)$ holds via 
the isomorphism  $(\ref{e1 Artin Wedderburn})$ in $(1)$ for every pair $(i, j)$ satisfying $i \neq j$. 
\end{enumerate}
\end{prop}

\begin{proof}
Let us show (1). 
By the Artin--Wedderburn theorem, we obtain the isomorphism (\ref{e1 Artin Wedderburn}), where 
$\ell_i$ is a positive integer and 
$\kappa_i$ is a division ring for every $i$. 
Note that 
\[
Z(\Hom(F, F)) \simeq Z(\kappa_1) \times \cdots \times 
Z(\kappa_n).    
\]
By $\dim_{\kappa} \Hom(F, F) < \infty$, 
we get $[\kappa_i : \kappa] = \dim_{\kappa} \kappa_i <\infty$. 

It suffices to show that each $\kappa_i$ is commutative. 
Note that 
$\kappa_i$ is a central simple algebra over a field $\kappa_i' :=Z(\kappa_i)$ satisfying $\kappa \subset \kappa'_i \subset \kappa_i$. 
Since $\kappa$ is a $C_1$-field, so is $\kappa'_i$. 
We then get  $\Br(\kappa'_i)=0$ \cite[Thoerem 1.2.12]{CTS21}, which implies 
that $\kappa_i = \kappa'_i$. 
Thus (1) holds. 
The assertion (2) follows from Lemma \ref{l Hom decompo} below. 
\qedhere



\end{proof}

\begin{lem}\label{l Hom decompo}
Let $X$ be a projective scheme over a field $\kappa$. 
For a coherent sheaf $F$, assume that we have 
a ring isomorphism 
\[
\Hom(F, F) = M_1 \times M_2, 
\]
where each $M_i$ is a (possibly non-commutative) nonzero subring 
of $\Hom(F, F)$. 
Then there exist nonzero coherent subsheaves $F_1$ and $F_2$ of $F$ such that $F = F_1 \oplus F_2$, 
$\Hom(F_1, F_2) = \Hom(F_2, F_1)=0$, and 
$\Hom(F_i, F_i) = M_i$ for each $i \in \{1, 2\}$. 
\end{lem}

\begin{proof}
For $\pi_1 := (1, 0), \pi_2 := (0, 1) \in M_1 \times M_2 = \Hom (F, F)$, 
we set 
\[
F_1 := \pi_1(F) \qquad\text{and}\qquad 
F_2 := \pi_2(F). 
\]
Note that 
$\pi_1^2 = \pi_1, \pi_2^2 = \pi_2$, 
$\pi_1\pi_2 = \pi_2\pi_1 =0$, and $\pi_1 + \pi_2 =1 (={\rm id}_F)$. 

Let us show that $F_1 + F_2 = F$ and $F_1 \cap F_2 = 0$. 
The former one follows from $F_1 + F_2 =\pi_1(F) + \pi_2(F)= (\pi_1 + \pi_2)(F) = {\rm id}_F(F) = F$. 
To prove $F_1 \cap F_2 = 0$, 
fix an affine open subset $U$ of $X$ and 
pick an element $y \in F_1(U) \cap F_2(U)$. 
We can write $y = \pi_1(x_1) =\pi_2(x_2)$ 
for some $x_1, x_2 \in F(U)$. 
Then it holds that 
\[
\pi_1(x_1) = \pi_1^2(x_1) = \pi_1(\pi_1(x_1)) = \pi_1(\pi_2(x_2)) =0. 
\]
Hence $F_1 \cap F_2 = 0$. 
Therefore, we get $F = F_1 \oplus F_2$. 

Take an element $\varphi =(\varphi_1, \varphi_2) \in M_1 \times M_2 = \Hom(F, F)$, where $\varphi_1 \in M_1$ and $\varphi_2 \in M_2$. 
We then get 
\[
\varphi \pi_1 = (\varphi_1, \varphi_2)(1, 0)  = (\varphi_1, 0) = (1, 0)  (\varphi_1, \varphi_2)= \pi_1\varphi. 
\]
Similarly, we obtain $\varphi \pi_2 = \pi_2 \varphi$. 
Therefore, we have 
\[
\varphi = \varphi \pi_1 + \varphi \pi_2 = 
\varphi \pi_1^2 + \varphi\pi_2^2 = \pi_1 \varphi \pi_1 + \pi_2 \varphi \pi_2. 
\]
For the natural inclusion $\iota_i : F_i \hookrightarrow F$, 
we set $\psi_i := \pi_i \varphi \iota_i: F_i \to F_i$. 
For  an affine open subset $U$ of $X$, $z_1 \in F_1(U)$, 
and $z_2 \in F_2(U)$, 
we get 
\[
\varphi(z_1, z_2) = 
( \pi_1 \varphi \pi_1 + \pi_2 \varphi \pi_2)(z_1, z_2) 
= \pi_1\varphi (z_1, 0) + \pi_2\varphi(0, z_2) = 
\psi_1(z_1) + \psi_2(z_2). 
\]
Hence $\Hom(F_1, F_2)=0$ and $\Hom(F_2, F_1)=0$. 
In particular, we get 
\[
\Hom(F, F) \xrightarrow{\simeq} \Hom(F_1, F_1) \times \Hom(F_2, F_2), \qquad \varphi \mapsto ( \pi_1  \varphi \iota_1, 
\pi_2 \varphi \iota_2). 
\]

It is enough to prove $\Hom(F_i, F_i) = M_i$   for each $i \in \{1, 2\}$. 
For $\varphi = (\varphi_1, 0) \in M_1 = M_1\times \{0\} \subset M_1 \times M_2 = \Hom(F, F)$, 
we get 
\[
\varphi =( \varphi_1, 0) = (1, 0) (\varphi_1, 0) = \pi_1 \varphi. 
\]
Then $\pi_2 \varphi \iota_2 = \pi_2 (\pi_1 \varphi) \iota_2 =0$, 
which implies  $\Im(\varphi) = \Im(\pi_1 \varphi) \subset \Im(\pi_1) 
=F_1$. 
Hence we obtain $M_1 \subset \Hom(F_1, F_1)$. 
By symmetry, we get $M_2 \subset \Hom (F_2, F_2)$. 
Then it holds that 
\[
\Hom(F_1, F_1) \times \Hom(F_2, F_2) = \Hom(F, F) = M_1 \times M_2 \subset \Hom(F_1, F_1) \times \Hom(F_2, F_2), 
\]
which implies $M_i = \Hom(F_i, F_i)$ for each $i \in \{1, 2\}$, as required.
\qedhere
\end{proof}

\begin{cor}\label{c Artin Wedderburn}
Let $X$ be a projective scheme  over a $C_1$-field $\kappa$ and 
let $F$ be a coherent sheaf on $X$. 
Then one and only one of the following holds. 
\begin{enumerate}
\item 
There exists an $\MO_X$-linear endomorphism $\varphi : F \to F$ such that 
$\Ker(\varphi) \neq 0$ and $\Im(\varphi) \neq 0$. 
\item 
A $\kappa$-algebra isomorphism 
\[
\Hom(F, F) \simeq \kappa'
\]
holds for some field $\kappa'$ which is a finite  extension of $\kappa$. 
\end{enumerate}
\end{cor}

\noindent 
If $\kappa$ is algebraically closed, then the condition (2) is equivalent to $\dim_\kappa \Hom(F, F) =1$.

\begin{proof}
If (2) holds, then every nonzero endomorphism $\varphi : F \to F$ is an automorphism, and hence (1) does not hold.  
Therefore, at least one of (1) and (2) fails. 
In what follows, we prove that (1) or (2) holds.

If there exists a nonzero nilpotent element of $\Hom(F, F)$, 
then (1) holds. 
In what follows, we assume that 
$\Hom(F, F)$ has no nonzero nilpotent element. 
Then $\Hom(F, F)$ is semisimple 
\cite[Definition 4.7, Theorem 4.12, Theorem 4.14]{Lam01}. 
By Proposition \ref{p Artin Wedderburn}(1), 
\[
\Hom(F, F) \simeq M_{\ell_1}(\kappa_1) \times \cdots \times M_{\ell_n}(\kappa_n) 
\]
for some positive integers $\ell_1, ..., \ell_n$ and 
finite field extensions $\kappa_1, ..., \kappa_n$ of $\kappa$. 
As $\Hom(F, F)$ has no nonzero nilpotent element, 
we get $\ell_1 = \cdots = \ell_n =1$. 
If $n \geq 2$, then we have $F = F_1 \oplus F_2$ for some nonzero coherent subsheaves $F_1$ and $F_2$ of $F$ (Proposition \ref{p Artin Wedderburn}(2)). 
Then (1) holds. 
If $n=1$, then (2) holds. 
\qedhere


\end{proof}



\subsection{Lifts and general members}

\begin{dfn}
\begin{enumerate}
\item Let $W$ be a complete discrete valuation ring. 
We say that $R$ is a {\em finite extension} of $W$ if 
$R$ is a complete discrete valuation ring 
such that $W$ is a subring of $R$ and $R$ is a finitely generated $W$-module. 
\item Given a finite extension $R$ of the ring $W(k)$ of Witt vectors, its residue field is equal to $k$. 
For an $R$-scheme $\cY$, we set $\cY_k := \cY \times_{\Spec R} \Spec k$. 
\end{enumerate}
\end{dfn}


\begin{dfn}
Fix a finite extension $R$ of $W(k)$. 
Let $Y$ be a projective scheme over $k$ and let $S$ be a closed subscheme of $Y$. 
\begin{enumerate}
\item
We say that $\cY$ is an $R$-{\em lift} of $Y$  if 
$\cY$ is a flat projective $R$-scheme satisfying $\cY_k \simeq Y$. 
\item 
We say that $\cS \subset \cY$ is an $R$-{\em lift} of $S \subset Y$ 
if $\cY$ and $\cS$ are flat projective $R$-schemes, 
$\cS$ is a closed subscheme of $\cY$, and there exists the following digram in which every equare is cartesian:  
\[
\begin{tikzcd}
S \arrow[r] \arrow[d, hook] & \cS \arrow[d, hook] \\
Y \arrow[r] \arrow[d]  & \cY\arrow[d]\\
\Spec k \arrow[r] & \Spec R,
\end{tikzcd}
\]
where the vertical arrows and $\Spec k \to \Spec R$ are the induced morphisms. 
\item 
Let $P$ be a closed point $Y$. 
We say that $\cP \subset \cY$ is an 
{\em $R$-lift} of $P \in Y$ if 
$\cP \subset \cY$ is an $R$-lift of $\{P\} \subset Y$ such that 
the induced composite morphism $\cP\hookrightarrow \cY \to \Spec R$ is an isomorphism. 
In particular, $\cP \subset \cY$ is a section of $\cY \to \Spec R$. 
\end{enumerate}
\end{dfn}

\begin{lem}
Fix a finite extension $R$ of $W(k)$. 
Let $Y$ be a smooth projective variety over $k$ and let $\cY$ be an 
$R$-lift of $Y$. 
Take a closed point $P \in Y$. 
Then there exists an $R$-lift $\cP \subset \cY$ of $P \in Y$. 
\end{lem}

\begin{proof}
Since $\cY \to \Spec R$  is smooth, it is formally smooth  \cite[\href{https://stacks.math.columbia.edu/tag/02H6}{Tag02H6}]{SP}. 
Hence the $k$-point $P\in Y(k)$ lifts 
successively to compatible $R/\m^n$-points of $\cY$. 
\end{proof}

\begin{prop}\label{p general lift}
Fix a finite extension $R$ of $W(k)$.  
Let $Y$ be a smooth projective variety over $k$ and let  
$\cY$ be an $R$-lift of $Y$. 
Set $Y_Q := \cY \times_R Q$ for $Q := \Frac R$. 
Then the following hold: 
\begin{enumerate}
\item An $R$-lift of a general closed point of $Y$ is a general closed point of $Y_Q$. 
\item An $R$-lift of a very general closed point of $Y$ is a very general closed point of $Y_Q$. 
\end{enumerate}
\end{prop}

\noindent 
The precise statement of (1) is (1)' below. 
\begin{enumerate}
\item[(1)']
Fix a proper closed subset $Z_Q \subset Y_Q$. 
Then there exists a proper closed subset $Z \subset Y$ such that 
if $P \in Y \setminus Z$ is a closed point and $\cP \subset \mathcal Y$ is an $R$-lift of $P \in Y$, 
then it holds that $\cP_Q \in Y_Q \setminus Z_Q$.  
\end{enumerate}

\begin{proof}
Since the proofs of (1) and (2) are identical, we only prove (1)'. 
We equip $Z_Q$ with the reduced scheme structure. 
By the properness of Hilbert scheme or \cite[Chapter III, Proposition 9.8]{Har77}, 
we can find an $R$-flat closed subscheme $\mathcal Z \subset \cY$ satisfying $\cZ_Q = Z$. 
Set $Z := \mathcal Z_k$. By $\dim Z =\dim Z_Q < \dim Y_Q = \dim Y$, 
$Z$ is set-theoretically a proper closed subset of $Y$. 

Fix a closed point $P \in Y \setminus Z$ 
and an $R$-lift $\cP \subset \mathcal Y$ of $P \in Y$. 
We have $P_Q := \cP_Q \in \cY_Q$. 
Suppose that $P_Q \in Z_Q$. 
Then $\cP = \overline{P}_Q \subset \overline{Z}_Q \subset \mathcal Z$, 
which would lead to a contradiction $P = \cP_k \in \mathcal Z_k =Z$. 
\end{proof}


\section{Brill-Noether generality}\label{s BN general}

\subsection{Weak  Brill-Noether generality and primitivity of $-K_X|_S$}

\begin{prop}[Weak BN-generality]\label{p weak BN generality}
Let $X$ be a prime Fano threefold  of genus $g$ such that $|-K_X|$ is very ample. 
Let $S$ be a smooth member of $|-K_X|$ and set $H := -K_X|_S$. 
Let $M$ be a Cartier divisor on $S$. 
Then 
\[
h^0(M) h^0(H-M) \leq g+1. 
\]
\end{prop}

\begin{proof}
We now reduce the problem to the case when $|M|$ is base point free. 
Let $M = M' +E$ be the decomposition into the mobile part $M'$ and the fixed part $E$ of $|M|$. 
By \cite[Proposition 3.6 and Proposition 3.10]{FanoI}, $|M'|$ is base point free. 
We have 
\[
h^0(M') h^0(H-M') = 
h^0(M) h^0(H+E-M) \geq h^0(M) h^0(H-M). 
\]
Therefore, it is enough to show the assertion by assuming that $|M|$ is base point free.

Suppose that $h^0(M) h^0(H-M) > g+1$ and $|M|$ is base point free. 
It suffices to derive a contradiction. 
By $h^0(H) = g+1$, we get $H^0(M-H)=0$. 
Set $F := F_{X, S, M}$ (Definition \ref{d ele tf}), i.e., 
$F$ is the vector bundle given by 
\begin{equation}\label{e1 weak BN generality}
0 \to F \to H^0(M) \otimes_k \MO_X \to M \to 0.     
\end{equation}
By Proposition \ref{p ele tf stable}, 
$F$ is $\mu$-stable, and hence $\dim_k \Hom(F, F) =1$ (cf.\ Corollary \ref{c Artin Wedderburn}). 



In order to derive a contradiction, 
let us show $\dim \Hom(F, F) \geq 2$. 
By applying $\Hom(-, F)$ to (\ref{e1 weak BN generality}), we obtain 
\[
\Hom(F, F) \xrightarrow{\simeq} \Ext^1(M, F), 
\]
because we have $H^0(F)=H^1(F)=0$ (Proposition \ref{p ele tf}). 
Applying $\Hom(M, -)$ to (\ref{e1 weak BN generality}), we get 
\[
0 \to \Hom_{\MO_X}(i_*M, i_*M) \to \Ext^1(i_*M, F) \to 
\Ext^1(i_*M, H^0(M) \otimes \MO_X) \to \Ext^1_{\MO_X}(i_*M, i_*M). 
\]
We have $\Hom_{\MO_X}(i_*M, i_*M) \simeq \Hom_{\MO_S}(M, M) \simeq k$. 
It holds that 
\[
\Ext^1_{\MO_X}(i_*M, \MO_X) \simeq H^0(S, \cE xt^1(i_*M, \MO_X)) 
\simeq H^0(S, -K_X|_S -M) = H^0(S, H-M), 
\]
because Grothendieck duality  implies 
\[
R\cHom_{\MO_X}(i_*M, \omega_X) \simeq 
i_*R\cHom(M, \omega_S)[-1] \simeq i_*\cHom(M, \MO_S)[-1]. 
\]
Therefore, we get 
\[
\dim_k \Ext^1(i_*M, H^0(M) \otimes \MO_X)  = h^0(M)h^0(H-M). 
\]

It suffices to show $\dim \Ext^1_{\MO_X}(i_*M, i_*M) \leq g+1$. 
The local-to-glocal spectral sequence for Ext induces an exact sequence 
\[
0 \to H^1(  \cH om_{\MO_X}(i_*M, i_*M)) \to \Ext^1_{\MO_X}(i_*M, i_*M) 
\to H^0( \cE xt^1_{\MO_X}(i_*M, i_*M)). 
\]
We have $H^1(  \cH om_{\MO_X}(i_*M, i_*M)) \simeq H^1(S, \MO_S)=0$. 
We then get 
$\dim \Ext^1_{\MO_X}(i_*M, i_*M) \leq g+1$ by 
\[
h^0(X, \cE xt^1_{\MO_X}(i_*M, i_*M)) 
\overset{(\star)}{\simeq} h^0(X, i_*\MO_S(H)) \simeq h^0(S, \MO_S(H)) = g+1. 
\]
where $(\star)$ follows from Lemma \ref{l Ext^1 MM} below. 
\end{proof}

\begin{lem}\label{l Ext^1 MM}
Let $X$ be a variety over a field and 
let $S$ be a nonzero effective Cartier divisor 
with the induced closed immersion $i : S \hookrightarrow X$. 
Take a line bundle $M$ on $S$. 
Then 
\[
\cExt^1_{\MO_X}(i_*M, i_*M) \simeq i_*(\MO_X(S)|_S).
\]
\end{lem}

\begin{proof}
Set $I_S:=\MO_X(-S)$. 
By $Ri_* =i_*$ and derived adjunction, we have
\[
R\mathcal Hom_{\MO_X}(i_*M,i_*M) 
\simeq
i_*R\mathcal Hom_{\MO_S}(Li^*i_*M,M).
\]
We compute $Li^*i_*M$. The standard locally free resolution
$0\to I_S \xrightarrow{\alpha} \MO_X\to i_*\MO_S\to 0$ 
of $i_*\MO_S$ gives
\[
Li^*i_*M
\simeq 
i_*\MO_S \otimes_{\MO_X}^L i_*M 
\simeq 
\left[
M\otimes_{\MO_S} I_S/I_S^2
\xrightarrow{\beta}
M
\right],
\]
where 
$M$ (resp.\ $M\otimes_{\MO_S} I_S/I_S^2$) is placed in degree $0$ (resp.\ $-1$). 
Note that $\beta=0$, as $\beta$ is induced by $\alpha$. 
Hence
$Li^*i_*M
\simeq
M\oplus \left(M\otimes_{\MO_S} I_S/I_S^2\right)[1]$. 
Therefore, 
\[
R\mathcal Hom_{\MO_S}(Li^*i_*
M,M)
\simeq
R\mathcal Hom_{\MO_S}(M,M)
\oplus
R\mathcal Hom_{\MO_S}\left(M\otimes_{\MO_S} I_S/I_S^2,M\right)[-1].
\]
As $M$ and $I_S/I_S^2$ are invertible sheaves on $S$, 
we get $R\mathcal Hom_{\MO_S}(M,M) \simeq \mathcal Hom_{\MO_S}(M,M)$ and 
$R\mathcal Hom_{\MO_S}\left(M\otimes_{\MO_S} I_S/I_S^2,M\right) \simeq \mathcal Hom_{\MO_S}\left(M\otimes_{\MO_S} I_S/I_S^2,M\right) \simeq 
\MO_X(S)|_S$, 
which implies 
$\cExt^1_{\MO_X}(i_*M,i_*M) = 
i_*\cH^1(R\mathcal Hom_{\MO_S}(Li^*i_*M,M))
\simeq i_*(\MO_X(S)|_S).$
\end{proof}

\begin{rem}
As an alternative proof, 
we can prove Proposition \ref{p weak BN generality} 
by using \cite[Proposition 3.2]{Muk10}. 
\end{rem}

The following simplification of the proof in an earlier draft was suggested by Kanemitsu.

\begin{thm}\label{t primitive -K_X|_S}
Let $X$ be a prime Fano threefold  of genus $g$ such that 
$|-K_X|$ is very ample. 
Let $S$ be a smooth member of $|-K_X|$. 
Then $\MO_X(-K_X)|_S$ is primitive in $\Pic S$ 
$($i.e., there exist no pair $(d, M)$ 
consisting of an integer $d \geq 2$ and a Cartier divisor $M$ on $S$ 
satisfying $\MO_X(-K_X)|_S \simeq \MO_S(dM)$$)$. 
\end{thm}

\begin{proof}
Suppose that there exist an integer $d \geq 2$ and a Cartier divisor $M$ on $S$ such that $\MO_X(-K_X)|_S \simeq \MO_S(dM)$. 
Since $M^2 \in 2\Z$ by the Riemann-Roch theorem $\chi(S, M) = \chi(S, \MO_S) + \frac{1}{2}M^2$,  it holds that  
\[
2g -2  = (-K_X)^3 = (-K_X|_S)^2 = (dM)^2 = d^2M^2 \in 2d^2 \Z. 
\]
By $3 \leq g \leq 12$ and $g \neq 11$, 
we have the following three solutions: 
\begin{enumerate}
\item $(g, d, M^2) = (5, 2, 2)$. 
\item $(g, d, M^2) = (9, 2, 4)$. 
\item $(g, d, M^2) = (10, 3, 2)$. 
\end{enumerate}
For $n >0$, we have $h^0(nM) = \chi(nM) = \chi(\MO_S) + \frac{1}{2}(nM)^2  = \frac{n^2}{2}M^2 +2$. 
Then  the weak Brill-Noether generality $g + 1 \geq h^0(M) h^0( (d-1)M)$ (Proposition \ref{p weak BN generality}) leads to the following contradiction in any case:  
\begin{enumerate}
\item 
$6= g+1 \geq  h^0(M) h^0(M) = ( \frac{1}{2} M^2 + 2)^2 
= 9$. 
\item 
$10= g+1 \geq  h^0(M) h^0(M) = 
( \frac{1}{2} M^2 + 2)^2 = 4^2 =16$. 
\item 
$11  = g+1 \geq  h^0(M) h^0(2M) = 
( \frac{1}{2} M^2 + 2)( 2M^2 + 2) 
= 3 \cdot 6 =18$. 
\end{enumerate}
\end{proof}

\subsection{Galois semisimplicity of Lazarsfeld bundles}

\begin{prop}\label{p generic Lefschetz}
Let $X$ be a prime Fano threefold 
such that $|-K_X|$ is very ample. 
Let $S_1$ and $S_2$ be 
smooth members of $|-K_X|$ 
such that $C := S_1 \cap S_2$ is a smooth curve. 
Let $T_0$ be the generic member of the pencil $\Lambda$ generated by $S_1$ and $S_2$, 
which is a smooth prime divisor on $X_{\kappa_0} := X \times_k \kappa_0$ for $\kappa_0 := K(\P^1_k) = k(t)$. 
For the purely inseparable closure $\kappa$ of $\kappa_0$, 
we set $T := T_0 \times_{\kappa_0} \kappa$. 
Then $T$ is a geometrically integral smooth prime divisor on $X_{\kappa} := X \times_k \kappa$ such that the pullback group homomorphism 
\[
\varphi : \Pic X \to \Pic T
\]
is an isomorphism. 
\end{prop}

\begin{proof}
For the algebraic closure $\ol{\kappa}$ of $\kappa$ and 
the base change $T_{\ol{\kappa}} := T \times_\kappa \ol{\kappa}$, 
the relation $T \sim -K_{X_\kappa}$ implies that 
$T_{\ol{\kappa}} \sim -K_{X_{\ol{\kappa}}}$. 
By $\Pic X_{\ol{\kappa}} = \Z K_{X_{\ol{\kappa}}}$, 
$T_{\ol{\kappa}}$ is a prime divisor. 
Hence $T$ is a geometrically integral smooth prime divisor.


By $\Pic X \simeq \Z$, it is enough to show that 
the pullback map $\varphi : \Pic X \to \Pic T$ is surjective. 
Let $\sigma : Y \to X$ be the blowup along $C$. 
For the induced pencil $\pi: Y \to \P^1_k$ and its generic fibre $T_Y$, 
we get $T_Y \xrightarrow{\simeq} T_0$. 
It follows from \cite[Proposition 5.17]{Tan24} that the pullback map 
\[
\varphi_Y : \Pic Y \to \Pic T_Y
\]
is surjective. 
By $-K_Y \sim -\sigma^*K_X-E \sim \pi^*\MO_{\P^1}(1)$, 
we have 
\[
\Pic Y = \sigma^*\Pic X \oplus \Z E = 
  \sigma^*\Pic X \oplus \Z (\pi^*\MO_{\P^1}(1)). 
\]
As the pullback $\varphi_Y(\pi^*\MO_{\P^1}(1))$
of $\pi^*\MO_{\P^1}(1)$ on $T_Y (\simeq T_0)$ is zero, 
the composition 
\[
\varphi_0 : \Pic X \xrightarrow{\sigma^*} \Pic Y \xrightarrow{\varphi_Y} \Pic T_Y \simeq \Pic T_0
\]
is surjective. 
Since $\rho(T) = \rho(T_0)=1$ \cite[Proposition 2.4]{Tan18b} and $T_{\ol{\kappa}}$ is a K3 surface, we get $\Pic T \simeq \Z$. 
Therefore, it is enough to show that the pullback $K_X|_T$ is primitive, which follows from the fact that 
$K_X|_{T_{\ol{\kappa}}}$ is primitive 
(Theorem \ref{t primitive -K_X|_S}). 
\end{proof}

\begin{prop}\label{p semisimple1}
Let $X$ be a regular projective variety 
over a perfect field $\kappa$. 
Let $F$ be a vector bundle such that $\Hom(F, F)$ is a field $($and hence commutative$)$. 
Set $n := \dim_{\kappa} \Hom(F, F)$. 
Then the following hold$:$ 
\begin{enumerate}
\item 
$F_{\ol{\kappa}} := F \otimes_\kappa \ol{\kappa} = F_1 \oplus \cdots \oplus F_n$ 
for some vector bundles $F_1, ..., F_n$ on $X_{\ol{\kappa}} := X \times_{\kappa} \ol{\kappa}$ 
satisfying $\dim_{\ol{\kappa}} (F_i, F_j) =\delta_{ij}$. 
In particular, each $F_i$ is simple. 
\item 
$F_1, ..., F_n$ are $\Gal(\ol{\kappa}/\kappa)$-conjugate 
$($i.e., if $i\neq j$, then 
we have $\sigma_{ij}(F_i) = F_j$ 
for some $\sigma_{ij} \in \Gal(\ol{\kappa}/\kappa)$$)$. 
\item  $\rank F_1 = \cdots = \rank F_n$. 
\end{enumerate}
\end{prop}

\begin{proof}
For $\kappa' := \Hom(F, F)$, we get $\kappa$-algebra isomorphisms 
\begin{equation}\label{e1 semisimple1}    
\ol{\kappa}_1 \times \cdots \times \ol{\kappa}_n 
\simeq \kappa' \otimes_{\kappa} \ol{\kappa} = 
\Hom(F, F) \otimes_{\kappa} \ol{\kappa} \simeq 
\Hom(F_{\ol{\kappa}}, F_{\ol{\kappa}})
\end{equation}
for $F_{\ol{\kappa}} := F \otimes_{\kappa} \ol{\kappa}$ 
and $\ol{\kappa}_i := \ol{\kappa}$. 
By Proposition \ref{p Artin Wedderburn}, we can find 
locally free subsheaves $F_1, ..., F_n$ of $F_{\ol{\kappa}}$ 
satisfying $\dim_{\ol{\kappa}}\Hom(F_i, F_j) =\delta_{ij}$. 
Thus (1) holds.


Let us show (2). 
By symmetry, we treat the case when $i=1$ and $j=2$. 
We can directly check that 
the isomorphisms in (\ref{e1 semisimple1}) are  
$\Gal(\ol{\kappa}/\kappa)$-equivariant with respect to the following 
$\Gal(\ol{\kappa}/\kappa)$-actions: 
\begin{enumerate}
\renewcommand{\labelenumi}{(\roman{enumi})}
\item The 
$\Gal(\ol{\kappa}/\kappa)$-action on 
$\kappa' \otimes_{\kappa} \ol{\kappa} = 
\Hom(F, F) \otimes_{\kappa} \ol{\kappa}$ is induced by the natural one on $\ol{\kappa}$. 
Then the 
$\Gal(\ol{\kappa}/\kappa)$-action on 
$\ol{\kappa}_1 \times \cdots \times \ol{\kappa}_n$ is defined 
so that the ring isomorphism $\ol{\kappa}_1 \times \cdots \times \ol{\kappa}_n 
\simeq \kappa' \otimes_{\kappa} \ol{\kappa}$ 
is $\Gal(\ol{\kappa}/\kappa)$-equivariant. 
\item 
The  $\Gal(\ol{\kappa}/\kappa)$-action on 
$\Hom(F_{\ol{\kappa}}, F_{\ol{\kappa}})$ 
is given by $\varphi \mapsto \sigma \cdot \varphi := \sigma \circ \varphi \circ \sigma^{-1}$  
for $\sigma \in \Gal(\ol{\kappa}/\kappa)$ and 
$\varphi \in \Hom(F_{\ol{\kappa}}, F_{\ol{\kappa}})$. 
\end{enumerate}
By (i), 
the Galois group $\Gal(\ol{\kappa}/\kappa)$ permutes the direct product factors $\ol{\kappa}_1, ..., \ol{\kappa}_n$ transitively. 
In particular, there exists $\sigma_{12} \in \Gal(\ol{\kappa}/\kappa)$ (i.e., a 
$\kappa$-linear automorphism 
$\sigma_{12} : \ol{\kappa} \to \ol{\kappa}$) 
such that $\sigma_{12}(\ol{\kappa}_1) = \ol{\kappa}_2$. 
For $\pi_1 := (1, 0, ..., 0), \pi_2 := (0, 1, 0, ..., 0) \in \ol{\kappa}_1 \times \cdots \times \ol{\kappa}_n$, 
we get $\sigma_{12} \cdot \pi_1 = \pi_2$. 
By $F_1=\pi_1(F_{\ol{\kappa}})$ and $F_2 = \pi_2(F_{\ol{\kappa}})$ (cf.\ the proof of Lemma \ref{l Hom decompo}), 
it holds that   
\[
F_2 = \pi_2(F_{\ol{\kappa}}) = (\sigma_{12} \cdot \pi_1) (F_{\ol{\kappa}}) \overset{{\rm (ii)}}{=} 
(\sigma_{12} \circ \pi_1 \circ \sigma_{12}^{-1})(F_{\ol{\kappa}}) = 
\sigma_{12}(\pi_1(F_{\ol{\kappa}})) = \sigma_{12}(F_1). 
\]
Thus (2) holds. 
The assertion (3) follows from (2). 
\end{proof}

\begin{cor}\label{c semisimple1.5}
Let $S$ be a smooth K3 surface over a perfect $C_1$-field $\kappa$ such that $\Pic S \simeq \Z$. 
Let $F$ be a $\mu$-stable vector bundle. 
Set $n:= \dim_\kappa\Hom(F, F)$. 
Let $\ol{\kappa}$ be the the algebraic closure of $\kappa$.  
Then the following hold: 
\begin{enumerate}
\item 
There exist locally free subsheaves $F_1, ..., F_n$ 
of $F_{\ol{\kappa}}:= F \otimes_\kappa \ol{\kappa}$ on $S_{\ol{\kappa}} := S \times_\kappa \ol{\kappa}$ such that 
\[
F_{\ol{\kappa}}= F_1 \oplus \cdots \oplus F_n
\]
and $\dim_{\ol{\kappa}}\Hom(F_i, F_j) = \delta_{ij}$. 
In particular, each $F_i$ is simple.
\item $F_1, ..., F_n$ are $\Gal(\ol{\kappa}/\kappa)$-conjugate. 
\item For $H:=-c_1(F)$ and $H_i := -c_1(F_i)$, it holds that 
\begin{itemize}
\item $H_1^2 = \cdots = H_n^2$, 
\item $H_1 \cdot H = \cdots = H_n \cdot H$, 
\item $\rank F_1 = \cdots = \rank F_n$, and 
\item $s(F_1) = \cdots = s(F_n)$. 
\end{itemize}
\end{enumerate}
\end{cor}

\begin{proof}
As $F$ is $\mu$-stable and $\kappa$ is a $C_1$-field, $\Hom(F, F)$ is a commutative division ring, i.e., a field. 
By Proposition \ref{p semisimple1}, 
(1) and (2) hold. 
It is easy to see that (2) implies (3). 
\end{proof}



\begin{cor}\label{c semisimple2}
Let $X$ be a prime Fano threefold 
such that $|-K_X|$ is very ample. 
Let $S_1$ and $S_2$ be 
smooth members of $|-K_X|$ 
such that $C := S_1 \cap S_2$ is a smooth curve. 
Assume that there is a Cartier divisor $A$ on $C$ such that 
$|A|$ and $|K_C-A|$ are base point free. 
Then, after taking the base change to some larger algebraically closed field, 
there exist a smooth member $S$ of $|-K_X|$ containing $C$ 
and vector subbundles 
$F_1, ..., F_n$ of 
the vector bundle $F := F_{S, C, A}$ on $S$ 
such that the following hold for $n := \dim_k \Hom(F, F)$$:$
\begin{enumerate}
\item 
$F  = F_1 \oplus \cdots \oplus F_n$. 
\item 
$\dim_k\Hom(F_i, F_j) = \delta_{ij}$. In particular, 
each $F_i$ is simple. 
\item For $H_i := -c_1(F_i)$, it holds that 
\begin{itemize}
\item $H_1^2 = \cdots = H_n^2$, 
\item $H_1 \cdot C = \cdots = H_n \cdot C$, 
\item $\rank F_1 = \cdots = \rank F_n$, and 
\item $s(F_1) = \cdots = s(F_n)$. 
\end{itemize}
\end{enumerate}
\end{cor}


\begin{proof}
Let $\Lambda$ be the pencil generated by $S_1$ and $S_2$. 
Then the generic member $T_0$ of $\Lambda$ is a smooth prime divisor on $X_{\kappa_0}$ for $\kappa_0 := k(t)$. 
For the purely inseparable closure $\kappa := k(t^{1/p^{\infty}})$ of $\kappa_0 = k(t)$, 
take the base change $T := T_0 \times_{\kappa_0} \kappa$, 
which is a smooth prime divisor on $X_{\kappa} := X \times_k \kappa$. 
We have $\Pic T = \Z H_T$ for $H_T := -K_X|_T$ 
(Proposition \ref{p generic Lefschetz}). 
Then $F_T := F_{T, C_{\kappa}, A_{\kappa}}$ is a $\mu$-stable vector bundle (Proposition \ref{p ele tf stable}). 
For the base change $S := T \times_{\kappa} \ol{\kappa}$ to the algebraic closure $\ol{\kappa}$ of $\kappa$, 
the assertion holds by Corollary \ref{c semisimple1.5}. 
\end{proof}

\subsection{Proof of Brill-Noether generality}


\begin{nota}\label{n main setting}
Let $X$ be a prime Fano threefold 
such that $|-K_X|$ is very ample. 
Assume that there exist 
smooth members  $S$ and $S'$ of $|-K_X|$ such that the following hold: 
\begin{enumerate}
\item $C := S \cap S'$ is a smooth curve. 
\item There is a Cartier divisor $A$ on $C$ such that 
$\deg A \leq g-1$ 
and each of $|A|$ and $|K_C -A|$ is base point free. 
\item 
The vector bundle $F := F_{S, C, A}$ on $S$ satisfies 
\[
F = F_1 \oplus \cdots \oplus F_n
\]
for $n := \dim_k \Hom(F, F)$ 
and simple vector subbundles $F_1, ..., F_n$ of $F$. 
\item For $H_i := -c_1(F_i)$, it holds that 
\begin{itemize}
\item $H_1^2 = \cdots = H_n^2$, 
\item $H_1 \cdot C = \cdots = H_n \cdot C$, 
\item $r := \rank F_1 = \cdots = \rank F_n$, and 
\item $s:=s(F_1) = \cdots = s(F_n)$. 
\end{itemize}
\end{enumerate}
In particular, 
$H := H_1 + \cdots + H_n \sim -K_X|_S$. 
\end{nota}

\begin{prop}\label{p BN H_i bpf}
We use Notation \ref{n main setting}. 
Then the following hold: 
\begin{enumerate}
\item 
$h^0(A) = nr, h^1(A)=ns$, and $r\leq s$. 
\item $rs \leq \frac{1}{2} H_i^2 +1$. 
\item 
For each $i$, $|H_i|$ is base point free and 
$h^0(S, H_i) \geq 2$. 
\item $h^0(F) = h^1(F) = 0$ and $h^2(F) = \chi(F) = h^0(A) + h^1(A)$. 
\end{enumerate}    
\end{prop}

\begin{proof}
Let us show (1). 
We have 
\[
 h^0(A)= \rank (H^0(A) \otimes_k \MO_S) = 
\rank F= \sum_{i=1}^n \rank F_i = nr.  
\]
The additivity of Mukai vectors and Proposition \ref{p ele tf} imply that 
\[
ns = ns(F_i)  =s(F) = \rank F+ \frac{1}{2} c_1(F)^2 -c_2(F)  = 
h^0(A) + \frac{1}{2} H^2 - \deg A 
\]
\[
= h^0(A) + g-1 -\deg A = h^1(A). 
\]
By $\deg A \leq g-1$ and the Riemann--Roch theorem, we get  
\[
n(r-s) = \chi(A) = \deg A +1 -g \leq 0, 
\]
and hence $r \leq s$. Thus (1) holds. 
The assertion (2) follows from the fact that $F_i$ is simple  \cite[Lemma 2.1]{Sha}.

Let us show (3). 
Applying $\cHom(-, \MO_S)$ to the exact sequence $0 \to F \to H^0(A) \otimes \MO_S \to A \to 0$, we get 
\[
0 \to H^0(A)^{\vee} \otimes \MO_S \to F^{\vee} \to \MO_C(K_C-A) \to 0
\]
where the last term can be computable by Grothendieck duality: 
$\cExt^1(i_*A, \MO_S) \simeq 
\cHom(A, \omega_C) \simeq \MO_C(K_C-A).$ 
Since 
$H^0(A)^{\vee} \otimes \MO_S$ and $\MO_C(K_C-A)$ are globally generated, 
it follows from $H^1(\MO_S)=0$ that $F^{\vee}$ and each direct summand $F_i^{\vee}$ are globally generated. 
By $\det F_i^{\vee} \simeq \MO_S(H_i)$, $|H_i|$ is base point free. 
We then get $h^0(S, H_i) \geq \chi(S, H_i) = 2 + \frac{1}{2}H_i^2 \geq 2$. 
Thus (3) holds.



Let us show (4). 
It follows from  Proposition \ref{p ele tf} that  $H^0(F) = H^1(F)=0$. 
Applying  $\chi(S, -)$ to 
\[
0 \to F \to H^0(A) \otimes\MO_S \to A \to 0, 
\]
we get $\chi(A) + \chi(F) = \chi(  H^0(A) \otimes\MO_S ) = 2h^0(A)$, which implies $\chi(F) = h^0(A) + h^1(A)$. 
Thus (4) holds. 
\end{proof}

\begin{thm}\label{t main1}
Let $X$ be a prime Fano threefold 
such that $|-K_X|$ is very ample. 
Let $S_1$ and $S_2$ be 
smooth members of $|-K_X|$ 
such that $C := S_1 \cap S_2$ is a smooth curve. 
Let $A$ be a Cartier divisor on $C$ such that 
$|A|$ and $|K_C-A|$ are base point free. 
Then $h^0(A) h^1(A) \leq g$. 
\end{thm}

\begin{proof}
Suppose that $h^0(A) h^1(A) > g$. 
Let us derive a contradiction. 
After switching $A$ and $K_C-A$ if necessary, 
we may assume $\deg A \leq g-1$. 
Moreover, taking a suitable base change, 
there exist a smooth member $S$ of $|-K_X|$ containing $C$ 
and vector subbundles $F_1, ..., F_n$ of the vector bundle $F := F_{S, C, A}$ 
satisfying the properties (1)--(3) in Corollary \ref{c semisimple2}.
In what follows, we use Notation \ref{n main setting}. 
In particular, each $|H_i|$ is base point free 
(Proposition \ref{p BN H_i bpf}).  
We have $n\geq 2$, as otherwise we would get the following contradiction: 
\[
g < h^0(A) h^1(A) = rs \leq \frac{1}{2}H^2 +1 = g. 
\]

\setcounter{step}{0}

\begin{step}\label{s1 t main1}
It holds that $H_i^2 < 4$ for every $i$. 
\end{step}

\begin{proof}[Proof of Step \ref{s1 t main1}]
Suppose that $H_i^2 \geq 4$ for some (every) $i$. 
Since $H_i$ is nef and big, we get 
\[
h^0(H_i) =\chi(S, H_i) = \frac{1}{2}H_i^2 + 2 \geq 4. 
\]
By the weak BN-generality (Proposition \ref{p weak BN generality}), we 
obtain the following contradiction: 
\[
13 \geq g +1 \geq h^0(H_1)h^0(H_2+ \cdots + H_n) \geq 
h^0(H_1)h^0(H_2) \geq 4 \cdot 4 =16. 
\]
This completes the proof of Step \ref{s1 t main1}. 
\end{proof}


\begin{step}\label{s2 t main1}
It holds that $H_i^2 \neq 2$ for every $i$. 
\end{step}

\begin{proof}[Proof of Step \ref{s2 t main1}]
Suppose that  $H_i^2 =2$  for some (every) $i$. 
Then $h^0(H_i) = 3$ and 
Proposition \ref{p BN H_i bpf} implies 
\[
rs  \leq \frac{1}{2}H_1^2 +1 = 2. 
\]
By the weak BN-generality  (Proposition \ref{p weak BN generality}), we get 
\[
g+1 \geq 
h^0(H_1)h^0(H_2+ \cdots + H_n) \geq  h^0(H_1)h^0(H_2) =3 \cdot 3 =9. 
\]
Hence $g \geq 8$. 
On the other hand, we have 
\[
8 \leq g < h^0(A) h^1(A) = (nr)(ns) = n^2 rs \leq 2n^2, 
\]
which implies that  $n \geq 3$. 
Since each of $H_2$ and $H_3$ is nef and big, we get $H_2 \cdot H_3 >0$. 
Then it holds that 
\[
h^0(H_2+ \cdots +H_n) \geq h^0(H_2+H_3) 
= \frac{1}{2}(H_2+H_3)^2 + 2  
\]
\[
= \frac{1}{2}(H_2^2 + H_3^2 +2H_2 \cdot H_3) +2 \geq  \frac{1}{2} (2 + 2+ 2 \cdot 1) +2 = 5.
\]
Hence $h^0(H_1)h^0(H_2+ \cdots + H_n) \geq 3 \cdot 5 =15 > 12+1 \geq g+1$, 
which contradicts the weak BN-generality  (Proposition \ref{p weak BN generality}). 
This completes the proof of Step \ref{s2 t main1}. 
\end{proof}

\begin{step}\label{s3 t main1}
For every $i$, it holds that $H_i^2=0$, 
$r=s=1$, 
$\deg A = g-1$, and $n \geq 3$. 
\end{step}

\begin{proof}[Proof of Step \ref{s3 t main1}]
We get $H_i^2 =0$ by 
$H_i^2 \in 2\Z_{\geq 0}$, Step \ref{s1 t main1}, and Step \ref{s2 t main1}. 
We have $rs \leq \frac{1}{2} H_1^2 +1 =1$, which implies $r=s=1$. 
It follows from the Riemann-Roch theorem for $C$ 
that $0 = nr -ns = h^0(A) -h^1(A) = \chi(A) = \deg A  + 1-g$. Therefore, $\deg A = g-1$. 


Suppose that $n <3$, i.e., $n=2$. 
Then $h^0(A)=h^1(A)=2$, which implies $g < h^0(A) h^1(A) =4$. 
Therefore, $g=3$ and $\deg A =2$. 
By $\deg A = h^0(A) = 2$, the induced morphism $\varphi_{|A|} : C \to \P^1_k$ is a double cover. 
Then $C$ is a hyperelliptic curve, which contradicts the fact that $|K_C|$ is very ample \cite[Ch.\ V, Proposition 5.2]{Har77}. 
This completes the proof of Step \ref{s3 t main1}. 
\end{proof}

Since $H$ is nef and big, $H$ is $1$-connected. 
After permuting $H_1, ..., H_n$ if necessary, we may assume $H_2 \cdot H_3 >0$. 
Then $H_2 + H_3+ \cdots +H_n$ is nef and big. 
Set $A' := H_1|_C$. 
By $H^{<2}(S, -(H_2+ \cdots +H_n))=0$, 
the restriction map 
\[
H^0(S, \MO_S(H_1)) \to H^0(C, \MO_S(H_1)|_C) =H^0(C, A')
\]
is an isomorphism. Then $h^0(C, A') = h^0(S, H_1) \geq 2$. 
By $n \geq 3$, it holds that 
\begin{equation}\label{e1 s3 t main1}
\deg A' = \frac{2g -2}{n} < g-1. 
\end{equation}

\begin{step}\label{s4 t main1}
$h^0(A') h^1(A') \leq g$. 
\end{step}

\begin{proof}[Proof of Step \ref{s4 t main1}]
Suppose 
$h^0(A') h^1(A') > g$. 
Note that $|K_C -A'| = |(H_2+ \cdots + H_n)|_C|$ is base point free, and hence $A'$ satisfies the assumption of Theorem \ref{t main1}. 
Then the above argument is applicable after replacing $A$ by $A'$. 
Hence the conclusion of Step \ref{s3 t main1} for $A'$ implies $\deg A' = g-1$, 
which contradicts (\ref{e1 s3 t main1}). 
This completes the proof of Step \ref{s4 t main1}. 
\end{proof}

\begin{step}\label{s5 t main1}
It holds that $n  <4$. 
\end{step}

\begin{proof}[Proof of Step \ref{s5 t main1}]
Suppose that $n \geq 4$. 
Then we get 
\[
h^1(A') = h^0(A') + g-1 - \deg A' \geq  2 +g-1 - \frac{2g-2}{n} 
\geq g+1 - \frac{2g-2}{4} = \frac{g+3}{2}
\]
and 
\[
h^0(A') h^1(A') \geq 2 \cdot  \frac{g+3}{2} = g+3 >g, 
\] 
which contradicts Step \ref{s4 t main1}. 
This completes the proof of Step \ref{s5 t main1}. 
\end{proof}

\begin{step}\label{s6 t main1}
It holds that $n  \neq 3$.  
\end{step}

\begin{proof}[Proof of Step \ref{s6 t main1}]
Suppose that $n =3$. 
Then $g < h^0(A) h^1(A) = n^2 = 9$. 
It holds that 
\[
h^1(A') = h^0(A') + g-1 - \deg A' \geq  2 +(g-1) - \frac{2g-2}{3} 
= \frac{g+5}{3} 
\]
and 
\[
h^0(A') h^1(A') \geq 2 \cdot  \frac{g+5}{3} = \frac{2g+10}{3} >g,
\] 
where the last inequality is guaranteed by $g <9$. 
This contradicts Step \ref{s4 t main1}. 
This completes the proof of Step \ref{s6 t main1}. 
\end{proof}
Step \ref{s3 t main1},
Step \ref{s5 t main1}, and 
Step \ref{s6 t main1} yield a contradiction. 
This completes the proof of Theorem \ref{t main1}.     
\end{proof}

\begin{thm}\label{t main2}
Let $X$ be a prime Fano threefold 
such that $|-K_X|$ is very ample. 
Then the following hold. 
\begin{enumerate}
\item If $S$ is a smooth member of $|-K_X|$, 
then $(S, -K_X|_S)$ is BN-general. 
\item If $S_1$ and $S_2$ are smooth members of $|-K_X|$ 
such that $C := S_1 \cap S_2$ is a smooth curve, then 
$C$ is BN-general. 
\end{enumerate}
\end{thm}

\begin{proof}
Since (2) implies (1) \cite[Theorem 2.10(1)]{BKM24}, let us show (2).
Let $C$ be a smooth curve as in the statement. 
Take a divisor $A$ on $C$. 
It suffices to show that $h^0(A) h^1(A) \leq g$. 
By the same argument as in \cite[Lemma 2.4]{Sha}, 
we may assume that $|A|$ and $|K_C-A|$ are base point free. 
Then we get $h^0(A) h^1(A) \leq g$ by Theorem \ref{t main1}. 
\qedhere


\end{proof}

\begin{rem}
In characteristic zero, 
the existence of a BN-general curve of the form $C =S_1 \cap S_2$ on a prime Fano threefold follows from \cite{Laz86} and 
the fact that 
a very general member $S$ of $|-K_X|$ satisfies $\Pic S = \Z (-K_X|_S)$. 
However, the latter property (i.e., $\Pic S = \Z (-K_X|_S)$) fails for the Fermat quartic threefold 
\[
X := \{ x_0^4+x_1^4+x_2^4+x_3^4+x_4^4=0\} \subset \P^4_k
\]
of characteristic three. 
Take an arbitrary smooth member $S$ of $|-K_X|$, which is a hyperplane section of $X$. 
It suffices to show that  $S$ is supersingular (i.e., $\rho(S)=22$). 
Otherwise, $S$ is quasi-$F$-split \cite[Theorem 4.4]{Yob19}. 
Then $X$ is quasi-$F$-split by \cite[Theorem 4.6]{KTTWYY1}, 
which contradicts \cite[Example 4.29]{KTY-Fedder2}. 
\end{rem}

\section{Partial simultaneous resolution via minimal model program}


\begin{lem}\label{l simul resol singularities}
Let $(R, \m)$ be a discrete valuation ring 
of mixed characteristic such that $\kappa := R/\m$ is algebraically closed. 
Let $\pi : Z \to \Spec R$ be a flat projective morphism, 
where $Z$ is an integral normal scheme. 
Assume that the closed fibre $S:=Z_\kappa$ is an integral scheme 
such that the normalisation $S^N$ of $S$ is Gorenstein and strongly $F$-regular. 
Then 
the following are equivalent for the induced morphism $h: S^N \to Z$:  
\begin{enumerate}
\item $S$ is normal. 
\item 
$(Z, S)$ is plt and $h^*(K_Z+S) \equiv K_{S^N}$. 
\end{enumerate}
\end{lem}

\begin{proof}
Assume (2). 
Then 
we have 
${\rm Diff}_{S^N}(0)=0$, 
where 
${\rm Diff}_{S^N}(0)$ denotes the different, 
which is an effective $\Q$-divisor on $S^N$ satisfying  $h^*(K_Z+S) = K_{S^N} + {\rm Diff}_{S^N}(0)$ 
\cite[Proposition 4.5]{Kol13}. 
Since $(S^N, 0)$ is strongly $F$-regular, 
$(S^N, 0)$ is $\text{BCM}_B$-regular \cite[Corollary 6.23]{MS21}. 
Then $S$ is normal by  \cite[Theorem A]{MSTWW}. 
Thus (1) holds.

Conversely, assume (1). 
Then $Z$ is Gorenstein around $S$. 
As the Gorenstein locus is open, $Z$ is Gorenstein. 
By the adjunction formula, we get  $(K_Z+S)|_{S} \sim K_{S}$. 
Then  $(Z, S)$ is purely $\text{BCM}_{B \to C}$-regular again 
by  \cite[Corollary 6.23]{MS21} and \cite[Theorem A]{MSTWW}. 
Therefore, $(Z, S)$ is plt  \cite[Theorem 5.4]{MSTWW}. 
Thus (2) holds. 
\end{proof}


\begin{prop}\label{p simul resol 1}
Let $(R, \m)$ be a discrete valuation ring of mixed characteristic such that $\kappa := R/\m$ is algebraically closed. 
Let $\pi : Z \to \Spec R$ be a flat projective morphism, 
where $Z$ is an integral normal scheme and 
every geometric fibre of $\pi$ is a canonical K3 surface. 
Assume that 
\begin{itemize}
\item the generic fibre $Z_\eta$ is smooth, and 
\item the fibre $Z_\kappa$ over $\kappa$ has a unique singular point and it is of type $A_1$. 
\end{itemize}
Then there exists a projective birational morphism 
$\varphi : Z' \to Z$ such that 
\begin{enumerate}
\item $Z'$ is a factorial integral normal scheme such that $(Z', Z'_\kappa)$ is plt, 
\item $\varphi_\eta : Z'_\eta \to Z_\eta$ is an isomorphism, and 
\item $\varphi_\kappa : Z'_\kappa \to Z_\kappa$ is 
either an isomorphism or the minimal resolution of $Z_{\kappa}$. 
\end{enumerate}
In particular, $\varphi : Z' \to Z$ is small. 

\end{prop}

\begin{proof}
Set $S := Z_\kappa$.
By Lemma \ref{l simul resol singularities}, $(Z, S)$ is plt. 
Let $\mu : V \to Z$ be a log resolution of $(Z, S)$. 
Let $\Ex(\mu) = \bigcup_{i \in I} E_i$ be the irreducible decomposition. 
For $E_V := \sum_{i \in I} E_i$, we have 
\[
K_V + S_V + E_V =\mu^*(K_Z+S) + \sum_{i\in I} a_i E_i 
\equiv_{\mu} \sum_{i\in I} a_i E_i, 
\]
where 
$a_i >0$ for every $i \in I$ and $S_V$ denotes the proper transform of $S$. 
We run a $(K_V+S_V+E_V)$-MMP over $Z$, whose existence is guaranteed by \cite[Theorem 6.2]{TY23}. 
By the negativity lemma \cite[Proposition 2.1]{TY23}, 
every $\mu$-exceptional prime divisor $E_i$ is contracted by this MMP. 
In particular, for its end result $(Z', S')$, 
the induced birational morphism $\varphi : Z' \to Z$ is 
small. 
Over the generic point $\eta \in \Spec R$, 
this MMP is a $(K_{V_\eta}+E_{\eta})$-MMP, and hence $\varphi_{\eta} : Z'_{\eta} \xrightarrow{\simeq} Z_{\eta}$. Thus (2) holds. 

By construction, we have $V_\kappa \subset \Supp(S_V + E_V)$. 
Every two-dimensional irreducible component of $V_\kappa$ except for $S_V$ 
is contracted by this MMP. 
As the closed fibre $Z'_\kappa$ is connected, 
$Z'_\kappa$ is a union of $S'$ and 
the one-dimensional irreducible components. 
On the other hand, 
$Z'_\kappa =\pi'^{-1}(\m)$ is an effective Cartier divisor on $Z'$ for the induced morphism $\pi' : Z' \to \Spec R$. 
Therefore, we get a set-theoretic equality $Z'_\kappa = S'$. 
Moreover, this equality is a scheme-theoretic one, because 
$Z'_\kappa = \varphi^{-1}(Z_\kappa) = \varphi^{-1}(S) = S'$.

Since $\varphi : Z' \to Z$ is small and $K_Z$ is Cartier, we get  
$K_{Z'} \sim \varphi^*K_Z$ and $K_{Z'}$ is Cartier. 
Since $K_{Z'} + S' \sim \varphi^*(K_Z+S)$ and $(Z, S)$ is plt, 
$K_{Z'}+S'$ is Cartier and $(Z', S')$ is plt. 
In particular,  the normalisation $S'^N$  of $S'$ is klt. 
We have 
\[
K_{S'^N} + D  \equiv (K_{Z'}+S')|_{S'^N} \equiv \varphi^*( (K_Z+S)|_S)  \equiv \psi^*K_S, 
\]
for the different $D :={\rm Diff}_{S'^N}(0) \geq 0$ \cite[Proposition 4.5]{Kol13} 
and the induced birational morphism $\psi :S'^N \to S$.  
Since $S$ is canonical, we get $D=0$. 
Then $S'$ is normal by Lemma \ref{l simul resol singularities}. 
As $K_{S'} \sim \psi^*K_S$, 
either $S' \xrightarrow{\psi, \simeq} S$ or $\psi$ is the minimal resolution of $S$. 
Thus (3) holds. 

Let us show that $Z'$ is factorial. 
Note that the non-regular locus of $Z'$ is zero-dimensional. 
As $Z'_\kappa$ is an effective Cartier divisor on $Z'$ having only hypersurface singularities, 
$Z'$ has only isolated hypersurface singularities. 
Then every $\Q$-Cartier Weil divisor is Cartier by \cite[Theorem 1.1.3(a)]{CS24}. 
This, together with $\Q$-factoriality of $Z'$, implies that $Z'$ is factorial. 
Thus (1) holds. 
\qedhere



\end{proof}

\begin{prop}\label{p simul resol 2}
Let $(R, \m)$ be a discrete valuation ring 
of mixed characteristic such that $\kappa := R/\m$ is algebraically closed. 
Let $\pi : Z \to \Spec R$ be a flat projective morphism, 
where $Z$ is an integral normal scheme. 
Assume that
every geometric fibre $Z_t$ of $\pi$ is a canonical K3 surface, 
$Z_t$ has a unique singular point $P$, and 
$P$ is of type $A_1$. 
Then 
there exists a projective birational morphism $\varphi :Z' \to Z$ 
such that 
$Z'$ is regular and 
$\varphi_t : Z'_t \to Z_t$ is the minimal resolution of $Z_t$ for every geometric point $t$ of $\Spec R$. 
\end{prop}

\begin{proof}
Set $S := Z_\kappa$. 
Let $\mu : V \to Z$ be a log resolution of $(Z, S)$. 
For the irreducible decomposition $\Ex(\mu) := \bigcup_{i=0}^r E_i$, we set $E_V := \sum_{i=0}^r E_i$. 
As $(Z, S)$ is plt (Lemma \ref{l simul resol singularities}), 
we have 
\[
K_V + S_V + E_V =\mu^*(K_Z+S) + \sum_{i =0}^r a_i E_i, 
\]
where 
$a_i >0$ for every $i \in \{0, ..., r\}$ and $S_V$ is the proper transform of $S$. 
Over the generic point $\eta \in \Spec R$, we have 
\[
K_{V_\eta} + E_{V_\eta} = \mu^*_{\eta} K_{Z_\eta} + \sum a_i E_{i, \eta}. 
\]
Since 
the geometric generic fibre $Z_{\overline{\eta}}$ 
has a unique $A_1$-singularity, 
we may assume that 
$a_0 = 1$  after permuting the indices if necessary. 
We have 
\[
K_V + S_V + (E_V-E_0) =\mu^*(K_Z+S) + \sum_{i \geq 1} a_i E_i
\]
and we run a $(K_V+S_V+E_V-E_0)$-MMP over $Z$, whose existence is guaranteed by \cite[Theorem 6.2]{TY23}. 
Let $Z'$ be the end result. 
Then the induced projective birational morphism 
$\varphi : Z' \to Z$ has a unique $\varphi$-exceptional prime divisor $E'$, 
which is the push-forward of $E_0$. 
Then $\varphi_\eta : Z'_\eta \to Z_\eta$ is nothing 
but the minimal resolution of $Z_\eta$. 

Set $S' := Z'_\kappa$
Let us show that $\varphi_\kappa : S'  \to S$ is the minimal resolution of $S$. 
Since $E'$ is not contained in the closed fibre $Z'_{\kappa}$, 
it follows from $K_{Z'_{\eta}} =\varphi^*_\eta K_{Z_\eta}$ 
that $K_{Z'} = \varphi^*K_Z$ and 
the divisor  $S' = \pi'^*\m = \varphi^*S$ is a prime Cartier divisor. 
For the induced birational morphism $\psi : S'^N \to S$ and the different $D :={\rm Diff}_{S'^N}(0) \geq 0$ \cite[Proposition 4.5]{Kol13}, we get 
\[
K_{S'^N} + D = (K_{Z'}+S')|_{S'^N} = \varphi^*(K_Z+S)|_{S'^N} = \psi^*K_S. 
\]
As $S$ is canonical, we obtain $D=0$ and $\psi : S'^N \to S$ is either an isomorphism or the minimal resolution of $S$. 
Since $(Z', S')$ is plt, $S'$ is normal (Lemma \ref{l simul resol singularities}). 
By the upper semi-continuity of fibre dimensions, 
$\psi : S' = Z'_{\kappa} \to Z_\kappa=S$ is not an isomorphism. 
Therefore, $\psi : S' \to S$ is the minimal resolution of $S$. 

As $S'$ is a regular effective Cartier divisor on $Z'$, $Z'$ is regular around $S'$. 
Since the open subset $Z'_\eta$ is regular as well, $Z'$ is regular everywhere. 
\qedhere



\end{proof}

\begin{thm}\label{t simul resol}
Let $(R, \m)$ be a discrete valuation ring of mixed characteristic such that $R/\m$ is algebraically closed. 
Let $\pi : Z \to \Spec R$ be a flat projective morphism, 
where $Z$ is an integral normal scheme. 
Assume that every geometric fibre $Z_t$ of $\pi$ is a canonical K3 surface satisfying 
one of the following: 
\begin{itemize}
\item $Z_t$ is smooth. 
\item $Z_t$ has a unique singular point and it is of type $A_1$. 
\end{itemize}
Then there exists a projective birational morphism 
$\varphi : Z' \to Z$ from an integral normal scheme $Z'$ 
such that 
\begin{enumerate}
\item $Z'$ is factorial, and 
\item 
if  $t \to \Spec R$  is a geometric point of $\Spec R$, 
then the induced morphism $\varphi_t : Z'_t \to Z_t$ between the fibres is either an isomorphism or the minimal resolution of $Z_t$.  
\end{enumerate}
\end{thm}

\begin{proof}
Set $\kappa :=R/\m$. 
If the fibre $Z_\kappa$ over $\kappa$ is smooth, 
then $\pi : Z \to \Spec R$ is a smooth morphism, and hence the assertion holds for 
the identity morphism $\varphi := {\rm id}_Z$. 
Hence we assume that $Z_\kappa$ is singular. 
If $Z_\eta$ is smooth (resp.\ singular), then the assertion follows from 
Proposition \ref{p simul resol 1} (resp.\ Proposition \ref{p simul resol 2}). 
\end{proof}

\section{Special Mukai divisors on K3 surfaces}
\label{s special Mukai div}


\subsection{Special Mukai divisors and minimal resolutions}



\begin{dfn}\label{d Mukai div}
We work over an algebraically closed field of arbitrary characteristic. 
Fix integers $r, s \geq 2$ and let $(S, H)$ be a 
quasi-polarised canonical K3 surface of genus $g = r \cdot s$. 
We say that $M$ is a {\em weakly special Mukai divisor of type $(r, s)$} 
if $M$ is a Cartier divisor on $S$ such that 
\begin{enumerate}
\item $M \cdot H = (r-1)(s+1)$, 
\item $h^0(\MO_S(M))=r$, and 
\item $H^1(\MO_S(M))= H^2(\MO_S(M))=0$. 
\end{enumerate}
We say that $M$ is a {\em special Mukai divisor of type $(r, s)$} 
if $M$ is a weakly special Mukai divisor of type $(r, s)$ such that $|M|$ is base point free. 
\end{dfn}

As we will see, base point freeness of $|M|$ is automatic in our applications. 

\begin{lem}\label{l wMukai div SD}
We work over an algebraically closed field of arbitrary characteristic. 
Fix integers $r, s \geq 2$ and let $(S, H)$ be a BN-general 
quasi-polarised canonical K3 surface of genus $g = r \cdot s$. 
Let $M$ be a Cartier divisor on $S$. 
Then $M$ is a weakly special Mukai divisor of type $(r, s)$ if and only if 
$H-M$ is a weakly special Mukai divisor of type $(s, r)$. 
\end{lem}

\begin{proof}
Assume that $M$ is a weakly special Mukai divisor of type $(r, s)$. 
By symmetry, it is enough to prove that the divisor $H-M$ 
satisfies the conditions (1)--(3) of Definition \ref{d Mukai div}. 
We have 
\[
(H - M) \cdot H = 2rs - 2 - (r - 1)(s + 1) = (r + 1)(s - 1) > 0
\]
In particular, Definition \ref{d Mukai div}(1) holds for $H-M$. 
Moreover, we get $h^2(\MO_S(H-M)) = h^0(\cO_S(M - H)) = 0$. 
The natural exact sequence
\begin{equation*}
0 \to \cO_S(M - H) \to \cO_S(M) \to \cO_C(M\vert_C) \to 0
\end{equation*}
shows that $h^0(\cO_C(M\vert_C)) \ge h^0(\cO_S(M)) = r$.
It holds that 
\[
h^1(\cO_C(M\vert_C))  =h^0(\cO_C(M\vert_C))  
- \deg(\cO_S(M)\vert_C) + g-1 \geq r  -(r - 1)(s + 1) +rs -1=s. 
\]
The above exact sequence implies that
\begin{equation*}
h^0(\cO_S(H - M)) = h^2(\cO_S(M - H)) = h^1(\cO_C(M\vert_C)) \ge s.
\end{equation*}
By BN-generality of $S$, we get 
$rs = g \geq h^0(M) h^0(H-M) \geq rs$, which implies 
$h^0(H-M) =s$, and hence Definition \ref{d Mukai div}(2) holds for $H-M$. 
It follows from  the Riemann--Roch theorem that 
\[
s -h^1(\MO_S(H-M)) = \chi(\MO_S(H-M)) = 2 + \frac{1}{2} (H-M)^2 
\]
\[
= 2+ \frac{1}{2}M^2 
+ \frac{1}{2}(H^2 -2H \cdot M) = \chi(\MO_S(M)) +
(H - M)\cdot H - \frac{1}{2}H^2 
\]
\[
= r + (r+1)(s-1) - \frac{1}{2} (2rs-2) = s. 
\]
Therefore, we get $H^1(\MO_S(H-M))=0$. Thus (3) holds. 
\end{proof}

\begin{lem}\label{l wMukai minuscule}
We work over an algebraically closed field of arbitrary characteristic. 
Fix integers $r, s \geq 2$ and let $(S, H)$ be a BN-general 
quasi-polarised smooth K3 surface of genus $g = r \cdot s$ 
such that $E := \bE(H)$ is a prime divisor 
$($cf.\ Subsection \ref{ss notation}$($\ref{def exc locus mathbbE}$))$. 
Let $M$ be a weakly special Mukai divisor of type $(r, s)$. 
Then $M$ is minuscule, that is, $M \cdot E \in \{1, 0, -1\}$. 
\end{lem}

\begin{proof}
It follows from Lemma \ref{l wMukai div SD} 
that $H-M$ is a weakly special Mukai divisor of type $(s, r)$. 
By $(H -M) \cdot E = -M \cdot E$, it suffices to show that $M \cdot E \leq 1$. 
Suppose $M \cdot E \geq 2$. 
Then $(M + E) \cdot E  \geq 2 + (-2)= 0$, which implies 
$H^0(\MO_E(M+E)) >0$. 
The exact sequence
\begin{equation*}
0 \to \cO_S(M) \to \cO_S(M +E) \to \cO_E(M + E) \to 0,
\end{equation*}
together with $H^1(\MO_S(M))=0$ (Definition \ref{d Mukai div}), 
implies that $h^0(\cO_S(M + E)) > h^0(\cO_S(M)) = r$.

On the other hand, we have $(H - M) \cdot E  = - M \cdot E \leq -2$, 
which implies $H^0(\cO_E(H - M)) = 0$. 
The exact sequence
\begin{equation*}
0 \to \cO_S(H - M - E) \to \cO_S(H - M) \to \cO_E(H - M) \to 0
\end{equation*}
implies that $h^0(\cO_S(H - M - E)) = h^0(\cO_S(H - M)) = s$.
We conclude that
\begin{equation*}
h^0(\cO_S(M + E)) \cdot h^0(\cO_S(H - M - E)) > rs = g, 
\end{equation*}
contradicting  BN generality of~$(S,H)$. 
Therefore, $M$ is minuscule.
\end{proof}

\begin{prop}\label{p weakly is not weakly}
We work over an algebraically closed field of arbitrary characteristic. 
Fix integers $r, s \geq 2$ and let $(S, H)$ be a BN-general 
quasi-polarised smooth K3 surface of genus $g = r \cdot s$ 
such that $E := \bE(H)$ is either empty or a prime divisor. 
Take a weakly special Mukai divisor $M$ on $S$ of type $(r, s)$. 
Let $\mu : S \to \ol{S}$ be the birational contraction such that 
$\Ex(\mu)=E$. 
Then the following hold:  
\begin{enumerate}
\item 
$|M|$ is base point free, that is, 
$M$ is a special Mukai divisor $M$ on $S$ of type $(r, s)$. 
\item 
The push-forward $\ol{M} := \mu_* M$ is 
a special Mukai divisor of type $(r, s)$ on $\ol{S}$ 
satisfying $M= \mu^* \ol{M}$. 
\end{enumerate}
\end{prop}

\begin{proof}
In what follows, we only treat the case when $E$ is a prime divisor, as otherwise the proof is simpler. 

\setcounter{step}{0}

\begin{step}\label{s1 weakly is not weakly}
In order to prove the assertion of Proposition \ref{p weakly is not weakly}, we may assume that $M \cdot E \geq 0$. 
\end{step}

\begin{proof}[Proof of Step \ref{s1 weakly is not weakly}]
Assume that $M \cdot E \leq 0$. 
For $M' := H-M$, we get  $M' \cdot E = (H-M ) \cdot E = -M \cdot E \geq 0$. 
As we are assuming that the assertion holds for the case when $M \cdot E \geq 0$, 
we get $M' = \mu^*\mu_*M'$, which implies that $M' \cdot E  =0$. 
We then get $M \cdot E = -M' \cdot E =0$. 
This completes the proof of Step \ref{s1 weakly is not weakly}.
\end{proof}

In what follows, we assume $M \cdot E \geq 0$.

\begin{step}\label{s2 weakly is not weakly}
It holds that $R^{>0}\mu_*\MO_S(M)= R^{>0}\mu_*\MO_S(-M)=0$ and $\mu_*\MO_S(M) = \MO_{\ol{S}}(\mu_*M)$. 
\end{step}

\begin{proof}[Proof of Step \ref{s2 weakly is not weakly}]
By Lemma \ref{l wMukai minuscule}, we have $M \cdot E \in \{-1, 0, 1\}$. 
Then it holds that  $(\pm M - (K_S +\Delta)) \cdot E \geq -1 - \frac{1}{2} (-2) =0$ 
for $\Delta :=  \frac{1}{2}E$. 
Hence we get $R^{>0}\mu_*\MO_S(M)= R^{>0}\mu_*\MO_S(-M)=0$ 
by the birational Kawamata-Viehweg vanishing theorem for surfaces \cite[Theorem 10.4]{Kol13}. 
We then obtain  
\[
R\cHom (\mu_*\MO_S(M), \MO_{\ol{S}}) 
\simeq R\mu_* R\cHom(\MO_S(M), \mu^!\MO_{\ol{S}}) 
\]
\[
\simeq R\mu_* R\cHom(\MO_S(M), \MO_S) \simeq R\mu_*\MO_S(-M) \simeq \mu_*\MO_S(-M).  
\]
Thus $\mu_*\MO_S(M)$ is maximal Cohen-Macaulay. 
Then it is isomorphic to its double dual, which is nothing but $\MO_S(\mu_*M)$. 
This completes the proof of Step \ref{s2 weakly is not weakly}.
\end{proof}

\begin{step}\label{s3 weakly is not weakly}
$|M|$ is base point free. 
\end{step}

\begin{proof}[Proof of Step \ref{s3 weakly is not weakly}]
By $H^{<2}(\MO_S(M-H)) \simeq H^{>0}(\MO_S(H-M))^\vee=0$ (Lemma \ref{l wMukai div SD}), the exact sequence 
\[
0 \to \MO_S(M -H) \to \MO_S(M) \to \MO_C(M|_C) \to 0. 
\]
induces  $H^0(\MO_S(M)) \xrightarrow{\simeq} H^0(\MO_C(M))$. 
Similarly, we get $H^0(\MO_S(H-M)) \xrightarrow{\simeq} H^0(\MO_C(K_C-M))$. 
Since $C$ is BN-general \cite[Theorem 1.5]{Sha}, the equality 
\[
h^0(\MO_C(M)) h^0(\MO_C(K_C-M)) =h^0(\MO_S(M))h^0(\MO_S(H-M)) = rs =g, 
\]
together with \cite[the proof of Lemma 2.4]{Sha}, implies that $|M|_C|$ is base point free. 
Therefore, 
$|M|$ is base point free around $C$. 
As $\mu_*(C)$ is ample, 
 a curve $\Gamma$ contained in the 
the base locus $\Bs|M|$ must be $E$. 
This, together with $M \cdot E \geq 0$, implie that $M$ is nef.

If $|M|$ is not base point free, 
then \cite[Theorem 3.16(2)(5)]{FanoI} implies that 
$\Bs |M| = E$ and $M \cdot E = g-2 \geq 2$. 
This contradicts $M \cdot E \in \{ -1, 0, 1\}$. 
Therefore, $|M|$ is base point free. 
This completes the proof of Step \ref{s3 weakly is not weakly}.
\end{proof}

\begin{step}\label{s4 weakly is not weakly}
$\ol{M}$ is Cartier and $|\ol{M}|$ is base point free. 
\end{step}

\begin{proof}[Proof of Step \ref{s4 weakly is not weakly}]
By Step \ref{s3 weakly is not weakly}, 
we can find two members $M_1, M_2 \in |M|$ satisfying $M_1 \cap M_2 \cap E =\emptyset$. 
Then the Koszul sequence 
\[
0 \to \MO_S(-M) \to \MO_S \oplus \MO_S \to \MO_S(M) \to 0 
\]
induced by $M_1$ and $M_2$ is exact around $E$. 
By $R^1\mu_*\MO_S(-M)=0$ (Step \ref{s2 weakly is not weakly}), 
the induced sequence 
\[
0 \to \mu_*\MO_S(-M) \to \MO_{\ol{S}} \oplus \MO_{\ol{S}} 
\to \mu_*\MO_{S}(M) \to 0 
\]
is exact around the singular point $\mu(E)$. 
Then  the sheaf $\MO_{\ol{S}}(\mu_*M)$, 
which is equal to  $\mu_*\MO_S(M)$  (Step \ref{s2 weakly is not weakly}), 
is  globally generated  around the singular point $\mu(E)$. 
As $|M|$ is base point free (Step \ref{s3 weakly is not weakly}), 
$\MO_S(\mu_*M)$ is globally generated outside the exceptional locus  $\mu(E)$. 
Therefore, $\MO_{\ol{S}}(\ol{M})$ is globally generated. 
By Lemma \ref{l bpf implies Cartier} below, $\ol{M}$ is Cartier. 
This completes the proof of Step \ref{s4 weakly is not weakly}.
\end{proof}

\begin{step}\label{s5 weakly is not weakly}
It holds that $M = \mu^*\ol{M}$.     
\end{step}

\begin{proof}[Proof of Step \ref{s5 weakly is not weakly}]
Since $\ol{M}$ is Cartier, we have 
$M = \mu^*\ol{M} + nE$ for some $n  \in \Z$. 
As $M$ is minuscule (Lemma \ref{l wMukai minuscule}), 
the following holds: 
\[
\{ 1, 0, -1\} \ni M \cdot E = 
(\mu^*\ol{M} + n E) \cdot E = 
-2n \in 2\Z. 
\] 
Therefore, we obtain $n=0$, as required. 
This completes the proof of Step \ref{s5 weakly is not weakly}.
\end{proof}

By Step \ref{s3 weakly is not weakly} and Step \ref{s5 weakly is not weakly}, (1) and (2) hold, respectively. 
This completes the proof of Proposition \ref{p weakly is not weakly}. 
\end{proof}

\begin{lem}\label{l bpf implies Cartier}
We work over an algebraically closed field of arbitrary characteristic. 
Let $T$ be a projective normal surface and let $D$ be a Weil divisor such that $\MO_T(D)$ is globally generated. 
Then $D$ is Cartier. 
\end{lem}

\begin{proof}
There exists a section $s \in H^0(T, \MO_T(D))$ 
whose zero locus $Z(t)$ is disjoint from the singular locus of $T$. 
In other words, the Weil divisor 
\[
D' := D + {\rm div}(s)
\]
does not contain any singular points of $T$. 
Then $D'$ is Cartier, and hence so is $D$. 
\end{proof}

\begin{cor}\label{c exist Mukai div up to resol}
We work over an algebraically closed field of arbitrary characteristic. 
Fix integers $r, s \geq 2$ and let $(S, H)$ be a BN-general 
quasi-polarised smooth K3 surface of genus $g = r \cdot s$ 
such that $E := \bE(H)$ is a prime divisor. 
Let $\mu : S \to \ol{S}$ be the birational contraction of $E$ 
and let $\ol{H}$ be an ample Cartier divisor on $\ol{S}$ satisfying 
$\mu^*\ol{H} \sim H$. 
Then 
$(S, H)$ has a special Mukai divisor  of type $(r, s)$ if and only if 
$(\ol{S}, \ol{H})$ has a special Mukai divisor  of type $(r, s)$.  
\end{cor}

\begin{proof}
If $N$ is a special Mukai divisor on $(\ol{S}, \ol{H})$ of type $(r, s)$, 
then its pullback $\mu^*N$ is a special Mukai divisor  on $(S, H)$ of type $(r, s)$. 
Conversely, if $M$ is  a special Mukai divisor on $S$ of type $(r, s)$, 
then its push-forward $\mu_*M$ is a special Mukai divisor on $\ol{S}$ of type $(r, s)$ (Proposition \ref{p weakly is not weakly}(2)). 
\end{proof}

\subsection{Existence of K3 surfaces having special Mukai divisors}

\begin{lem}\label{l Mukai div limit}
Let $(R, \m)$ be a discrete valuation ring of mixed characteristic 
such that $\kappa := R/\m$ is algebraically closed. 
Fix integers $r, s \geq 2$. 
Let $\pi : Y \to \Spec R$ be a flat projective morphism, 
where $Y$ is an integral normal scheme. 
Let $H$ be a $\pi$-nef and $\pi$-big Cartier divisor on $Y$. 
Assume that the following hold: 
\begin{enumerate}
\item $Y$ is factorial. 
\item 
For every geometric point $t \to \Spec R$, 
$(Y_t, H|_{Y_t})$ is a  BN-general canonical quasi-polarised 
K3 surface such that either 
\begin{itemize}
\item $Y_t$ is smooth, or 
\item $Y_t$ has a unique singular point and it is of type $A_1$. 
\end{itemize}
\item There exists a Cartier divisor $M$ on the generic fibre $Y_\eta$ 
whose base change $M_{\ol{\eta}}$ 
is a  weakly special Mukai divisor of type $(r, s)$ on 
the geometric generic fibre $(Y_{\ol{\eta}}, H_{\ol{\eta}})$. 
\end{enumerate}
Then the fibre $(Y_{\kappa}, H_{\kappa})$ over $\kappa$ has a weakly special Mukai divisor of type $(r, s)$. 
\end{lem}

\begin{proof}
Let $M_Y$ be the closure of $M$, which is a Cartier divisor on $Y$ by (1). 
Set $M_\kappa := (M_Y)|_{Y_\kappa}$. 
It suffices to show that $M_\kappa$ satisfies the conditions (1)--(3) of 
Definition \ref{d Mukai div}. 
Definition \ref{d Mukai div}(1) holds for $M_\kappa$ 
by the invariance of intersection numbers under flat families. 
The upper semi-continuity \cite[Chapter III, Theorem 12.8]{Har77} implies that  
\[
h^0(\MO_{Y_\kappa}(M_\kappa)) \geq 
h^0(\MO_{Y_{\ol{\eta}}}(M_{\ol{\eta}})) =r, \quad 
h^0(\MO_{Y_\kappa}(H|_{Y_\kappa}- M_\kappa)) \geq 
h^0(\MO_{Y_{\ol{\eta}}}(H|_{Y_{\ol{\eta}}}- M_{\ol{\eta}})) =s
\]
If $h^0(\MO_{Y_\kappa}(M_\kappa)) >r$, 
then we would get 
$h^0(\MO_{Y_\kappa}(M_\kappa))h^0(\MO_{Y_\kappa}(H|_{Y_\kappa}- M_\kappa))> rs =g$, contradicting BN-generality. 
Hence Definition \ref{d Mukai div}(2) holds for $M_\kappa$. 
By $M_\kappa \cdot H_{\kappa} =(r-1)(s+1)>0$, we get $H^2(\MO_{Y_\kappa}(M_{\kappa}))=0$. 
Since $\chi(-)$ is invariant under flat families \cite[Chapter III, Theorem 9.9]{Har77}, we get 
\[
r = \chi(\MO_{Y_{\ol{\eta}}}(M_{\ol{\eta}})) = \chi(\MO_{Y_\kappa}(M_{Y_\kappa})) 
= r -h^1(\MO_{Y_\kappa}(M_{Y_\kappa})). 
\]
Therefore, Definition \ref{d Mukai div}(3) holds for $M_\kappa$. 
\end{proof}

\begin{thm}\label{t special Mukai div}
Let $X$ be a prime Fano threefold of genus $g \geq 8$. 
Fix integers $r$ and $s$ satisfying $g = rs$ and $s \geq r\geq 2$. 
Then there exist a canonical K3 surface $S \in |-K_X|$ and 
a special Mukai divisor $M$ on 
the polarised canonical K3 surface $(S, -K_X|_S)$ of type $(r, s)$ 
such that either $S$ is smooth, or $S$ has a unique singular point and it is of type $A_1$. 
\end{thm}

\begin{proof}
By standard argument, the problem is reduced to the case when $k$ is uncountable. 
Let a $R$ be a finite extension of $W(k)$. 
In what follows, we will replace $R$ by a further finite extension of $R$. 
Set $Q :=K(R)$. 
We have a lift $\mathcal X \subset \P^{g+1}_R$ of $X \subset \P^{g+1}_k$ over $R$ \cite[Theorem A]{KTLift1}. 
Then there exists a (family of) pencils 
$\cV \subset \P(H^0(\cX, -K_{\cX}))$ such that 
$V := \cV_k \subset \P(H^0(X, -K_X))$ and 
$V_{\ol{Q}} := \cV_{\ol{Q}} \subset \P(H^0(X_{\ol{Q}}, -K_{X_{\ol{Q}}}))$ 
are very general pencils (Proposition \ref{p general lift}). 
In particular, each of $V$ and $V_{\ol{Q}}$ is a Lefschetz pencil 
\cite[Theorem 3.6]{KTLift1}. 

Replacing $R$ by a finite extension if necessary, 
we can find a member $Z_Q$ of $V_Q$ (i.e., a $Q$-rational point of $V_Q$) 
and a Cartier divisor $M_{Q}$ on $Z_Q$ such that 
the pullback $M_{\ol{Q}}$ on the  base change 
$(Z_{\ol{Q}}, -K_{X_{\ol{Q}}}|_{Z_{\ol{Q}}})$  
is a special Mukai divisor of type $(r, s)$ \cite[Proposition 5.5]{BKM24}. 
Let $\pi : Z \to \Spec R$ be the flat projective morphism obtained by the specialisation. 
Then there exists a projective birational morphism $\varphi : Z' \to Z$ from an integral normal scheme $Z'$ as in Theorem \ref{t simul resol}.


Since 
 $Z_{\ol{Q}}$ has 
a weakly special Mukai divisor of type $(r, s)$, so does $Z'_{\ol{Q}}$ by Corollary \ref{c exist Mukai div up to resol}. 
Then 
Lemma \ref{l Mukai div limit} implies that 
$Z'_k$ has a weakly special Mukai divisor of type $(r, s)$, and hence so does $S := Z_k$ again by 
Corollary \ref{c exist Mukai div up to resol}. 
\qedhere

\end{proof}

\begin{rem}
Our proof of Theorem \ref{t special Mukai div} might suggest an  alternative way to prove 
the corresponding statement in characteristic zero. 
For example, assume that there exists a family $\cX \to U$ of prime Fano theefolds of fixed genus 
such that $U$ is irreducible and  the conclusion of Theorem \ref{t special Mukai div} holds for general fibres. 
Then the proof of Theorem \ref{t special Mukai div} might imply the conclusion for all fibres. 
\end{rem}

\section{Mukai bundles on Fano threefolds}\label{s Mukai bdl X}

\subsection{Existence of Mukai bundles}

\begin{dfn}\label{d Mukai bdl X}
Let $X$ be a prime Fano threefold $X$ of genus $g$. 
Take integers $r$ and $s$ satisfying $g=rs$, $r \geq 2$ and $s \geq 2$. 
A {\em Mukai bundle} $\cU_X$ on $X$ of type $(r,s)$ is 
a vector bundle satisfying the following properties: 
\begin{enumerate}
\item 
$\rank(\cU_X) = r$ and $c_1(\cU_X) = K_X$. 
\item 
$H^i(X,\cU_X) = 0$ for every $i \geq 0$.
\item 
$\cU_X^\vee$ is globally generated. 
\item 
$\dim H^0(X, \cU_X^\vee)=r+s$ and $H^{>0}(X, \cU_X^\vee) = 0$.
\end{enumerate}
For a Mukai bundle $\cU_X$ on $X$ of type $(r, s)$, 
we define $\U_X^\perp$ by the following exact sequence: 
\[
0 \to \U_X^\perp \to H^0(X, \cU_X^\vee) \otimes \MO_X 
\to \cU_X^\vee \to 0, 
\]
where 
$H^0(X, \cU_X^\vee) \otimes \MO_X 
\to \cU_X^\vee$ is the induced restriction homomorphism. 
\end{dfn}

\begin{lem}\label{l perp is Mukai}
Let $X$ be a prime Fano threefold $X$ of genus $g$. 
Take integers $r$ and $s$ satisfying $g=rs$, $r \geq 2$, and $s \geq 2$. 
Let $\cU_X$ be a Mukai bundle on $X$ of type $(r,s)$. 
Then $\cU_X^\perp$ is a Mukai bundle on $X$ of type $(s, r)$. 
\end{lem}

\begin{proof}
We have an exact sequence 
\[
0 \to \cU_X^\perp \to H^0(X, \cU_X^\vee) \otimes \MO_X 
\xrightarrow{\Ev} \cU_X^\vee \to 0. 
\]
As $H^0(X, \cU_X^\vee) \otimes \MO_X$ and $\cU_X^\vee$ are vector bundles, so is $\cU_X^\perp$. 
It suffices to check that  (1)-(4) of Definition \ref{d Mukai bdl X} holds for $\U_X^\perp$.

We have $c_1(\cU_X^\perp) = -c_1(\cU_X^\vee) = K_X$ 
and $\rank \cU_X^\perp = h^0(X, \cU_X^\vee)  -\rank \cU_X^\vee = (r+s)-r =s$. 
Thus Definition \ref{d Mukai bdl X}(1) holds  for $\U_X^\perp$. 
Note that $H^0(\Ev)$ is an isomorphism. 
This, together with $H^{>0}(X, \MO_X)= H^{>0}(X, \cU_X^\vee)=0$, implies that $H^i(X, \cU_X^\perp)=0$ 
for every $i \geq 0$. 
Thus Definition \ref{d Mukai bdl X}(2) holds  for $\U_X^\perp$.

Dualising the above exact sequence, we obtain 
\[
0 \to \cU_X \to H^0(X, \cU_X^\vee)^\vee \otimes \MO_X 
\to (\cU_X^\perp)^\vee \to 0. 
\]
As $H^{\bullet}(X, \cU_X)=0$ and $H^{>0}(X, \MO_X)=0$, 
we get $H^{>0}(X, (\cU_X^\perp)^\vee) =0$ and 
$h^0(X, (\cU_X^\perp)^\vee) = \dim H^0(X, \cU_X^\vee)^\vee =s+r$. 
Thus Definition \ref{d Mukai bdl X}(3) and Definition \ref{d Mukai bdl X}(4) hold  for $\U_X^\perp$. 
\end{proof}

\begin{thm}\label{t Mukai bdl X}
Let $X$ be a prime Fano threefold $X$ of genus $g \geq 8$. 
Take integers $r$ and $s$ satisfying $g=rs$, $r \geq 2$, and $s \geq 2$.
Then there exists a Mukai bundle $\cU_X$ on $X$ of type $(r,s)$. 
\end{thm}

\begin{proof}
By Lemma \ref{l perp is Mukai}, we may assume $s \geq r \geq 2$. 
Take a member $S \in |-K_X|$ and a special Mukai divisor $M$ on 
a polarised K3 surface 
$(S, -K_X|_S)$ of type $(r, s)$ as in Theorem \ref{t special Mukai div}. 
Set $\cU_X := F_{X, S, M}$ (Definition \ref{d ele tf}), 
which is defined by the following exact sequence: 
\[
0 \to \cU_X \to H^0(S, \MO_S(M)) \otimes \MO_X \xrightarrow{e} \MO_S(M) \to 0. 
\]
Then Definition \ref{d Mukai bdl X}(1) follows from Proposition \ref{p ele tf}. 
Note that we have $H^{>0}(S, M) =0$ and $H^{>0}(X, \MO_X)=0$. 
Since $H^0(e)$ is an isomorphism, we get $H^\bullet (X,\cU_X) = 0$, i.e., 
Definition \ref{d Mukai bdl X}(2) holds. 
Dualising the above exact sequence, we get 
\[
0 \to H^0(S, \MO_S(M))^{\vee} \otimes \MO_X \to \cU_X^{\vee} \to 
\MO_S(H -M) \to 0. 
\]
As $\MO_S(H-M)$ is globally generated, 
so is $\cU_X^{\vee}$ by $H^1(\MO_X)=0$. 
Thus Definition \ref{d Mukai bdl X}(3) holds.  
As $H-M$ is a special Mukai divisor of type $(s, r)$ (Lemmna \ref{l wMukai div SD}), 
we get 
$H^{>0}(S, \MO_S(H-M))=0$ and $h^0(S, \MO_S(H-M))=s$ 
(Definition \ref{d Mukai div}). 
This, together with $H^{>0}(X, \MO_X)=0$ and $h^0(S, \MO_S(M))=r$, 
implies $H^{>0}(X, \cU_X^\vee) =0$ and $h^0(X, \cU_X^\vee)=r+s$. 
Thus  Definition \ref{d Mukai bdl X}(4) holds. 
\end{proof}

\begin{rem}\label{r perpX vs perpS}
Let $X$ be a prime Fano threefold of genus $g \geq 8$. 
Let $\cU_X$ be a Mukai bundle of type $(r, s)$, 
where $r$ and $s$ are integers satisfying $g=rs$, $r \geq 2$, and $s \geq 2$. Take a  member $S$ of $|-K_X|$. 
Set $\cU_S := \cU_X|_S$ and we define $\cU_S^\perp$ by the exact sequence 
\[
0 \to \cU_S^\perp \to H^0(S, \cU_S^\vee) \otimes \MO_S 
\to \cU_S^\vee \to 0. 
\]
Then $\cU_X^\perp|_S \simeq \cU_S^\perp$. 
Indeed, this follows from 
$H^0(X, \cU_X^\vee) \xrightarrow{\simeq} H^0(S, \cU_S^\vee)$ and the exact sequence $0 \to \cU_X^\perp|_S \to H^0(X, \cU_X^\vee) \otimes \MO_S \to \cU_X^\vee|_S \to 0$. 
\end{rem}

\subsection{Slope stability}






\begin{prop}\label{p Mukai bdl stable}
Let $X$ be a prime Fano threefold of genus $g$. 
Let $\cU_X$ be a Mukai bundle of type $(r, s)$, 
where $r$ and $s$ are integers satisfying $g=rs$, $r \geq 2$, and $s \geq 2$.
Then $\cU_X$ is $\mu$-stable. 
\end{prop}

\begin{proof}
The assertion follows from  Definition \ref{d Mukai bdl X} and Lemma \ref{l stablity criterion}. 
\qedhere

\end{proof}

\begin{prop}\label{p restriction stable}
Let $X$ be a prime Fano threefold of genus $g$. 
Let $\cU_X$ be a Mukai bundle of type $(r, s)$, 
where $r$ and $s$ are integers satisfying $g=rs$, $r \geq 2$, and $s \geq 2$.
For a field extension $\kappa/k$, 
let $S$ be a smooth member of $|-K_{X_{\kappa}}|$. 
Set $\cU_S := \cU_X|_S$. 
Then the following hold. 
\begin{enumerate}
\item 
$\cU_S^\vee$ is globally generated, $H^0(S, \cU_S)=0$, and $v(\cU_S) = (r, K_X|_S, s)$
\item 
If $\Pic(S) = \Z (K_X|_S)$, then $\cU_S$ is $\mu$-stable. 
\end{enumerate}
\end{prop}

\begin{proof}
Let us show (1). 
As $\cU_X^{\vee}$ is globally generated, so is $\cU_S^{\vee} (= \cU_X^{\vee}|_S)$. 
By $H^0(\cU_X) = 0$ and $H^1(\cU_X(K_X)) = H^2(\cU_X^{\vee})=0$, 
the exact sequence 
\[
0 \to \cU_X(K_X) \to \cU_X \to \cU_S \to 0
\]
implies $H^0(S, \cU_S)=0$. 
Clearly, $\rank(\cU_S) = \rank(\cU_X) = r$ and $c_1(\cU_S) = c_1(\cU_X)|_S = K_X|_S$. 
Note that we have 
$H^i(X, \cU_X^{\vee}(K_X)) \simeq H^{3-i}(X, \cU_X)^{\vee} =0$. 
By 
\[
0 \to \cU_X^\vee(K_X) \to \cU_X^\vee \to \cU_S^\vee \to 0, 
\]
we get $H^i(X, \cU_X^{\vee}) \xrightarrow{\simeq} H^i(S, \cU_S^\vee)$ for every $i$. 
Thus $H^{>0}(S, \cU_S^\vee)=0$ and $h^0(S, \cU_S^\vee)=r+s$. 
The Riemann--Roch theorem and Serre duality imply
\[
r+s = \chi(S, \cU_S^\vee) = \chi(S, \cU_S) = 
2r + \frac{1}{2} c_1(\cU_S)^2 -c_2(\cU_S) 
\]
\[
= 2r + \frac{1}{2} c_1(\cU_S)^2 + (s(\cU_S) -r - \frac{1}{2}c_1(\cU_S)^2) 
= r+ s(\cU_S). 
\]
Hence (1) holds. 
The assertion (2) follows from (1) and Lemma \ref{l stablity criterion}. 
\end{proof}

\subsection{Uniqueness of Mukai bundles}

\begin{prop}\label{p uniqueness Mukai bdl0}
Let $X$ be a prime Fano threefold of genus $g$. 
Take integers $r$ and $s$ satisfying $g=rs$, $r \geq 2$, $s \geq 2$, and $(r, s) \neq (2, 2)$. 
Let $\cU_1$ and $\cU_2$ be vector bundles  such that the following hold: 
\begin{enumerate}
\item $h^0(X, \cU_1^\vee) =r+s$, 
$H^0(X, \cU_1^\vee(K_X))=0$, 
$\cU_1^\vee$ is globally generated, 
and 
$\cU_1^\perp$ is $\mu$-stable for the vector bundle $\cU_1^\perp$ defined by the exact sequence 
\[
0 \to \cU_1^\perp \to H^0(X, \cU_1^\vee) \otimes \MO_X \to \cU_1^\vee \to 0. 
\]
\item $\cU_2$ is a Mukai bundle of type $(r, s)$. 
\item There exists a smooth member $S \in |-K_X|$ such that $\cU_1|_S \simeq \cU_2|_S$. 
\end{enumerate}
Then $\cU_1 \simeq \cU_2$. 
\end{prop}

\begin{proof}
By (2) and (3), 
we have $\rank \cU_1 = \rank \cU_2=r$, 
and both  
$\cU_1$ and $\cU_2$ are $\mu$-stable of the same slope (Proposition \ref{p Mukai bdl stable}). 
Fix an isomorphism $\varphi : \cU_1^\vee|_S \xrightarrow{\simeq} \cU_2^\vee|_S$, whose existence is ensured by (3).


By the induced isomorphism $H^0(X, \cU_2^\vee) \xrightarrow{\simeq} H^0(S, \cU_2^\vee|_S)$, 
there exists a unique linear map 
$\alpha\colon
H^0(X,\cU_1^\vee)
\to 
H^0(X,\cU_2^\vee)
$ 
completing the following commutative diagram: 
\[
\begin{tikzcd}
H^0(X,\cU_1^\vee)
\arrow[r,"\alpha"]
\arrow[d]
&
H^0(X,\cU_2^\vee)
\arrow[d, "\simeq"]
\\
H^0(S,\cU_1^\vee|_S)
\arrow[r,"\simeq"', "H^0(\varphi)"]
&
H^0(S,\cU_2^\vee|_S).
\end{tikzcd}
\]
It follows from $H^0(X, \cU_1^\vee(K_X))=0$ 
that the left vertical restriction map 
$H^0(X, \cU_1^\vee) \to H^0(S, \cU_1^\vee|_S)$ is injective. 
This, together with $h^0(X, \cU_1^{\vee}) = r+s =h^0(X, \cU_2^{\vee})$, implies that this restriction map and $\alpha$ are isomorphisms. 
Then we get  the following commutative diagram: 
\[
\begin{tikzcd}[column sep={between origins,11.6em}]
\cU_1^\perp
\arrow[r,hook,"\iota"]
\arrow[d]
\arrow[rrr, bend left=17, "\psi"']
&
H^0(X,\cU_1^\vee)\otimes\cO_X
\arrow[r,"\alpha\otimes\operatorname{id}_{\MO_X}", "\simeq"']
\arrow[d]
&
H^0(X,\cU_2^\vee)\otimes\cO_X
\arrow[r,two heads,"\pi"]
\arrow[d]
&
\cU_2^\vee
\arrow[d]
\\
\cU_1^\perp|_S
\arrow[r,hook,"\iota|_S"]
\arrow[rdd, "0", bend right]
&
H^0(X,\cU_1^\vee)\otimes\cO_S
\arrow[r,"\alpha \otimes\operatorname{id}_{\MO_S}", "\simeq"']\arrow[d, "\simeq"]
&
H^0(X,\cU_2^\vee)\otimes\cO_S
\arrow[r,two heads,"\pi|_S"] \arrow[d, "\simeq"]
&
\cU_2^\vee|_S\\
&H^0(S,\cU_1^\vee|_S)\otimes\cO_S
\arrow[r,"H^0(\varphi) \otimes\operatorname{id}_{\cO_S}", "\simeq"']
\arrow[d]
&
H^0(S,\cU_2^\vee|_S)\otimes\cO_S
\arrow[d]
\\
&\cU_1^\vee|_S
\arrow[r,"\varphi"]
&
\cU_2^\vee|_S 
\arrow[ruu, equal, bend right]
\end{tikzcd}
\]
It holds that $\psi|_S=0$, because 
the second horizontal sequence is obtained by taking the restriction of the first one to $S$.
Therefore, $\psi$ factors through 
$\widetilde{\psi}\colon
\cU_1^\perp
\longrightarrow
\cU_2^\vee(K_X)$. 
If $\psi=0$, then $\psi$ induces a surjection $\cU_1^\vee \twoheadrightarrow \cU_2^\vee$, which is automatically an isomorphism.  
Thus it is enough to prove that $\wt{\psi}=0$. 
Since {$\cU_1^\perp$} and $\cU_2^\vee(K_X)$ are $\mu$-stable, 
it suffices to show $\mu_{-K_X}(\cU_1^\perp)>\mu_{-K_X}(\cU_2^\vee(K_X))$, which follows from 
$-\frac{1}{s} > -\frac{r-1}{r}$, 
$\mu_{-K_X}(\cU_1^\perp)
=-\frac{(-K_X)^3}{s}$, and 
\[
\mu_{-K_X}(\cU_2^\vee(K_X))
= \frac{ c_1( \cU_2^\vee(K_X)) \cdot (-K_X)^2}{r} = 
- \frac{r-1}{r} (-K_X)^3. 
\]




\end{proof}

\begin{lem}\label{l Mukai vs Laz}
Let $S$ be a smooth K3 surface over a field $\kappa$ such that 
$\Pic S = \Z H$ for an ample Cartier divisor $H$. 
Set $g := \frac{1}{2} H^2 +1$. 
Take integers $r$ and $s$ satisfying $g=rs$, $r \geq 2$, and $s \geq 2$. 
For each $i \in \{1, 2\}$, 
let $\cV_i$ be a $\mu$-stable vector bundle with $v(\cV_i) = (r, -H, s)$. 
Then $\cV_1 \simeq \cV_2$. 
\end{lem}

\begin{proof}
The same argument as in \cite[Lemma 3.2]{BKM24} works. 
\end{proof}

\begin{prop}\label{p uniqueness Mukai bdl}
Let $X$ be a prime Fano threefold of genus $g \geq 6$. 
Let $\cU_1$ and $\cU_2$ be Mukai bundles of type $(r, s)$, 
where $r$ and $s$ are integers satisfying $g=rs$, $r \geq 2$, and $s \geq 2$. 
Then $\cU_1 \simeq \cU_2$. 
\end{prop}

\begin{proof}
After replacing $X$ by a base change to a larger algebraically closed field, 
it is enough to find  a smooth member $S \in |-K_X|$ such that 
$\cU_1|_S \simeq \cU_2|_S$  (Proposition \ref{p uniqueness Mukai bdl0}).  
Lemma \ref{l Mukai vs Laz} is applicable for the generic member $T$ of $|-K_X|$ and the base field $\kappa$ of $T$, so that 
$\cU_1|_T \simeq \cU_2|_T$. 
Hence the base change $S :=T \times_\kappa \ol{\kappa}$ of $T$ to the algebraic closure $\ol{\kappa}$ is a required member of $|-K_{X_{\ol{\kappa}}}|$. 
\qedhere

\end{proof}

\begin{lem}\label{l BKM25 3.3}
Let $S$ be a smooth K3 surface over a field $\kappa$ such that 
$\Pic S = \Z H$ for an ample Cartier divisor $H$. 
Set $g := \frac{1}{2} H^2 +1$. 
Let $h : S \to \Gr(r, r+s)$ be a morphism such that 
$h^*\MO_{\Gr(r, r+s)}(1) \simeq \MO_S(H)$, 
where  $r$ and $s$ are integers satisfying $g=rs$, $r \geq 2$, and $s \geq 2$. 
Assume that $h(S)$ does not lie in the sub-Grassmannians
\[
\Gr(r, r+s-1) \subset \Gr(r, r+s), \qquad or \qquad 
\Gr(r-1, r+s-1) \subset \Gr(r, r+s). 
\]
Let $\cU$ $($resp.\ $\cQ$$)$ is the universal subbundle 
$($resp.\ the universal quotient bundle$)$ on $\Gr(r, r+s)$. 
Then the following hold: 
\begin{enumerate}
\item $h^*\cU$ and $h^*\cQ$ are $\mu$-stable vector bundles. 
\item $v(h^*\cU) =(r, -H, s)$ and $v(h^*\cQ) = (s, H, r)$. 
\end{enumerate}
\end{lem}

\begin{proof}
The same argument as in \cite[Lemma 3.3]{BKM25} works. 
\end{proof}

\begin{prop}\label{p uniqueness Mukai bdl1}
Let $X$ be a prime Fano threefold of genus $g \geq 6$. 
Let $f : X \to \Gr(r, r+s)$ be a morphism such that $f^*\MO_{\Gr(r, r+s)}(1) \simeq \MO_X(-K_X)$, 
where $r$ and $s$ are integers satisfying $g=rs$, $r \geq 2$, and $s \geq 2$. 
Assume that 
$f(X)$ does not lie in the sub-Grassmannians
\[
\Gr(r, r+s-1) \subset \Gr(r, r+s)\qquad or \qquad 
\Gr(r-1, r+s-1) \subset \Gr(r, r+s). 
\]
Set $\cU_1 := f^*\cU$ and let $\cU_2$ be a Mukai bundle of type $(r, s)$, 
where $\cU$ denotes the universal subbundle $\cU$ on $\Gr(r, r+s)$. 
Then $\cU_1 \simeq \cU_2$. 
\end{prop}

\begin{proof}
Let  $\cU_1^\perp$ be the vector bundle defined by the exact sequence $0 \to \cU_1^\perp \to H^0(X, \cU_1^\vee) \otimes \MO_X \to \cU_1^\vee \to 0$. 

\setcounter{step}{0}

\begin{step}\label{s1 uniqueness Mukai bdl1}
For a general member $S \in |-K_X|$, 
its image $f(S)$ does not lie in the sub-Grassmannians
\[
\Gr(r, r+s-1) \subset \Gr(r, r+s), \qquad or \qquad 
\Gr(r-1, r+s-1) \subset \Gr(r, r+s). 
\]
\end{step}

\begin{proof}[Proof of Step \ref{s1 uniqueness Mukai bdl1}]
We have $f : X \to \Gr(r, r+s)  =\Gr(r, V)$ for $V := k^{\oplus (r+s)}$. 

For a hyperplane $W \subset V$, we have the sub-Grassmannian $\Gr(r,W)\subset \Gr(r,V)$. 
Such sub-Grassmannians are parametrised by $\mathbb P(V^\vee) (\simeq \P^{r+s-1})$. 
Consider the incidence closed subset
\[
I_1
:=
\left\{
(S,W)\in |-K_X| \times\mathbb P(V^\vee)
\mathrel{}\middle|\mathrel{}
f(S)\subset\Gr(r,W)
\right\}, 
\]
where $|-K_X| \simeq \P^{g+1}$. 
For every $W \in \P(V^{\vee})$, 
$f^{-1}(\Gr(r, W))$ is a proper closed subset of $X$, and hence the induced morphism $\pi : I_1 \to \P(V^{\vee})$ is finite. 
Therefore, $\dim I_1 \leq \dim  \P(V^{\vee}) = r+s-1$. 
Since  
\[
\dim I_1 \leq r+s -1 < rs+1 =g+1 = \dim \P^{g+1} = \dim |-K_X|, 
\]
a general member $S \in |-K_X|$ is not contained in 
any sub-Grassmannian of the form $\Gr(r, W)$.

The same argument is applicable  to sub-Grassmannians of the form
$\Gr(r-1,V/L)\subset\Gr(r,V)$, 
where $L\subset V$ is a one-dimensional subspace. 
Indeed, such sub-Grassmannians 
are parametrised by $\mathbb P(V) (\simeq \P^{r+s-1})$.  
For a fixed $L \subset V$, the inverse image
$f^{-1}(\Gr(r-1,V/L))$ 
is a proper closed subset of $X$ by assumption, 
and hence contains only finitely many members of $|-K_X|$. 
Thus the corresponding incidence closed subset 
$I_2$ satisfies $\dim I_2 \leq r+s-1$. 
This completes the proof of Step \ref{s1 uniqueness Mukai bdl1}.    
\end{proof}

\begin{step}\label{s2 uniqueness Mukai bdl1}
$\cU_1$ and 
$f^*\cQ$
are $\mu$-stable. 
After replacing $X$ by the base change to a larger algebraically closed field, there exists a smooth member $S \in |-K_X|$ satisfying $\cU_1|_S \simeq \cU_2|_S$. 
\end{step}

\begin{proof}[Proof of Step \ref{s2 uniqueness Mukai bdl1}]
Let $T$ be the generic member of $|-K_X|$. 
We have $\Pic T = \Z H$  for $H :=-K_X|_T$. 
Moreover, the conclusion of Step \ref{s1 uniqueness Mukai bdl1} holds also for $T$. 
Hence Lemma \ref{l BKM25 3.3} is applicable for $T$, so that 
$f^*\cU|_T$ and $f^*\cQ|_T$ are $\mu$-stable. 
In particular, both $\cU_1 =f^*\cU$ and $f^*\cQ$ are $\mu$-stable.

Since both $\cU_1|_T = f^*\cU|_T$ and $\cU_2|_T$ are $\mu$-stable vector bundles with $v(\cU_1|_T) = v(\cU_2|_T)$ (Lemma \ref{l BKM25 3.3}), 
we get $\cU_1|_T \simeq \cU_2|_T$ (Lemma \ref{l Mukai vs Laz}). 
Therefore, the assertion holds by setting $S :=T \times_{\kappa} \ol{\kappa}$, which is the base change of $T$ to the algebraic closure $\ol{\kappa}$, where $\kappa$ denotes the base field of $T$. 
This completes the proof of Step \ref{s2 uniqueness Mukai bdl1}. 
\end{proof}

\begin{step}\label{s3 uniqueness Mukai bdl1}
$H^0(X, \cU_1^\vee(K_X))=0$, $h^0(X, \cU_1^\vee) =r+s$, and 
$\cU_1^{\perp}$ is $\mu$-stable. 
\end{step}

\begin{proof}[Proof of Step \ref{s3 uniqueness Mukai bdl1}]
We have $\Hom(\cU_1, \MO_X(K_X))=0$, 
because 
$\cU_1$ and $\MO_X(K_X)$ are $\mu$-stable and 
\[
\mu_{-K_X}(\cU_1)
=
-\frac{(-K_X)^3}{r}
>
-(-K_X)^3
=
\mu_{-K_X}(\MO_X(K_X)).
\]
Therefore, $H^0(X, \cU_1^\vee(K_X)) \simeq \Hom(\cU_1, \MO_X(K_X)) =0$. 
This, together with 
$h^0(S, \cU_1^\vee|_S) =h^0(S, \cU_2^\vee|_S) =h^0(X, \cU^\vee_2) =r+s$ and the exact sequence 
\[
0 \to \cU_1^\vee(K_X) \to \cU_1^\vee \to \cU_1^\vee|_S \to 0, 
\]
implies $h^0(X,\cU_1^\vee)\leq r+s$.

The exact sequence $0 \to
f^*\cQ^\vee
\to
(\MO_X^{\oplus(r+s)})^\vee \xrightarrow{\alpha}
f^*\cU^\vee
\to
0$ induces 
a linear map $H^0(\alpha): 
(k^{\oplus(r+s)})^\vee
\to
H^0(X,\cU_1^\vee)$. 
Note that $H^0(\alpha)$ is injective, as otherwise 
$f(X)$ would be contained in the sub-Grassmannian
\[
\Gr(r, \Ker(\tau))\subset\Gr(r,r+s)
\]
defined by a nonzero element 
$\tau \in 
(k^{\oplus(r+s)})^\vee$ satisfying $H^0(\alpha)(\tau)=0$. 
Then it holds that $h^0(X,\cU_1^\vee)\geq r+s$. 
Therefore,  $h^0(X,\cU_1^\vee) = r+s$ and $(f^*\cQ)^\vee \simeq \cU_1^{\perp}$. 
Thus $\cU_1^{\perp}$ is $\mu$-stable. 
This completes the proof of Step \ref{s3 uniqueness Mukai bdl1}. 
\end{proof}

\begin{step}\label{s4 uniqueness Mukai bdl1}
$\cU_1 \simeq \cU_2$. 
\end{step}

\begin{proof}[Proof of Step \ref{s4 uniqueness Mukai bdl1}]
By standard argument, we may replace $X$ by the base change to a larger algebraically closed field. 
Then we may assume that there exists a smooth member $S \in |-K_X|$ satisfying $\cU_1|_S \simeq \cU_2|_S$ (Step \ref{s2 uniqueness Mukai bdl1}). 
This, together with Step \ref{s3 uniqueness Mukai bdl1}, implies that 
Proposition \ref{p uniqueness Mukai bdl0} is applicable, so that 
$\cU_1 \simeq \cU_2$. 
This completes the proof of Step \ref{s4 uniqueness Mukai bdl1}. 
\end{proof}
Step \ref{s4 uniqueness Mukai bdl1} completes the proof of 
Proposition \ref{p uniqueness Mukai bdl1}. 
\end{proof}

\bibliographystyle{skalpha}
\bibliography{reference.bib}

@article{BKM24,
  author =        {Bayer, Arend and Kuznetsov, Alexander and Macrì, Emanuele},
  journal =       {arXiv:2402.07154},
  title =         {Mukai bundles on {F}ano threefolds},
  year =          {2024},
}

@article{BKM25,
  author =        {Bayer, Arend and Kuznetsov, Alexander and Macrì, Emanuele},
  journal =       {arXiv:2501.16157},
  title =         {Mukai models of {F}ano varieties},
  year =          {2025},
}

@book {CTS21,
    AUTHOR = {Colliot-Th\'{e}l\`ene, Jean-Louis and Skorobogatov, Alexei N.},
     TITLE = {The {B}rauer-{G}rothendieck group},
    SERIES = {Ergebnisse der Mathematik und ihrer Grenzgebiete. 3. Folge. A
              Series of Modern Surveys in Mathematics [Results in
              Mathematics and Related Areas. 3rd Series. A Series of Modern
              Surveys in Mathematics]},
    VOLUME = {71},
 PUBLISHER = {Springer, Cham},
      YEAR = {2021},
     PAGES = {xv+453},
      ISBN = {978-3-030-74247-8; 978-3-030-74248-5},
   MRCLASS = {14F22 (14E08 14G05 14G12 14K05)},
  MRNUMBER = {4304038},
MRREVIEWER = {Thomas Benedict Williams},
       DOI = {10.1007/978-3-030-74248-5},
       URL = {https://doi-org.utokyo.idm.oclc.org/10.1007/978-3-030-74248-5},
}

@article{DPS94,
  author  = {Demailly, Jean-Pierre and Peternell, Thomas and Schneider, Michael},
  title   = {Compact complex manifolds with numerically effective tangent bundles},
  journal = {J. Algebraic Geom.},
  volume  = {3},
  number  = {2},
  year    = {1994},
  pages   = {295--345},
}

@book {Fri98,
    AUTHOR = {Friedman, Robert},
     TITLE = {Algebraic surfaces and holomorphic vector bundles},
    SERIES = {Universitext},
 PUBLISHER = {Springer-Verlag, New York},
      YEAR = {1998},
     PAGES = {x+328},
      ISBN = {0-387-98361-9},
   MRCLASS = {14J60 (14-01 14J15 32J15 57R55)},
  MRNUMBER = {1600388},
MRREVIEWER = {I.\ Dolgachev},
       DOI = {10.1007/978-1-4612-1688-9},
       URL = {https://doi-org.kyoto-u.idm.oclc.org/10.1007/978-1-4612-1688-9},
}

@book {Har77,
    AUTHOR = {Hartshorne, Robin},
     TITLE = {Algebraic geometry},
    SERIES = {Graduate Texts in Mathematics, No. 52},
 PUBLISHER = {Springer-Verlag, New York-Heidelberg},
      YEAR = {1977},
     PAGES = {xvi+496},
      ISBN = {0-387-90244-9},
   MRCLASS = {14-01},
  MRNUMBER = {0463157},
MRREVIEWER = {Robert Speiser},
}

@article {Isk77,
    AUTHOR = {Iskovskih, V. A.},
     TITLE = {Fano threefolds. {I}},
   JOURNAL = {Izv. Akad. Nauk SSSR Ser. Mat.},
  FJOURNAL = {Izvestiya Akademii Nauk SSSR. Seriya Matematicheskaya},
    VOLUME = {41},
      YEAR = {1977},
    NUMBER = {3},
     PAGES = {516--562, 717},
      ISSN = {0373-2436},
   MRCLASS = {14J10 (14M20 14N05)},
  MRNUMBER = {463151},
MRREVIEWER = {Miles Reid},
}

@article {Isk78,
    AUTHOR = {Iskovskih, V. A.},
     TITLE = {Fano threefolds. {II}},
   JOURNAL = {Izv. Akad. Nauk SSSR Ser. Mat.},
  FJOURNAL = {Izvestiya Akademii Nauk SSSR. Seriya Matematicheskaya},
    VOLUME = {42},
      YEAR = {1978},
    NUMBER = {3},
     PAGES = {506--549},
      ISSN = {0373-2436},
   MRCLASS = {14J10 (14J30 14M20 14N05)},
  MRNUMBER = {503430},
MRREVIEWER = {Miles Reid},
}

@incollection {IP99,
    AUTHOR = {Iskovskikh, V. A. and Prokhorov, Yu. G.},
     TITLE = {Fano varieties},
 BOOKTITLE = {Algebraic geometry, {V}},
    SERIES = {Encyclopaedia Math. Sci.},
    VOLUME = {47},
     PAGES = {1--247},
 PUBLISHER = {Springer, Berlin},
      YEAR = {1999},
   MRCLASS = {14J45 (14E07 14F22 14K30)},
  MRNUMBER = {1668579},
MRREVIEWER = {Takao Fujita},
}

@article{KT-primitive,
  author =        {Kanemitsu, Akihiro and Tanaka, Hiromu},
  journal =       {preprint},
  title =         {Prime Fano threefolds in positive characteristic},
  year =          {2026},
}

@article {KTTWYY1,
    AUTHOR = {Kawakami, Tatsuro and Takamatsu, Teppei and Tanaka, Hiromu and
              Witaszek, Jakub and Yobuko, Fuetaro and Yoshikawa, Shou},
     TITLE = {Quasi-{$F$}-splittings in birational geometry},
   JOURNAL = {Ann. Sci. \'Ec. Norm. Sup\'er. (4)},
  FJOURNAL = {Annales Scientifiques de l'\'Ecole Normale Sup\'erieure.
              Quatri\`eme S\'erie},
    VOLUME = {58},
      YEAR = {2025},
    NUMBER = {3},
     PAGES = {665--748},
      ISSN = {0012-9593,1873-2151},
   MRCLASS = {14E05 (13A35 14G17)},
  MRNUMBER = {4962159},
}

@article{KTY-Fedder2,
  author =        {Kawakami, Tatsuro and Takamatsu, Teppei and Yoshikawa, Shou},
  journal =       {arXiv:2511.17270},
  title =         {Fedder type criteria for quasi-{F}-splitting {II}},
  year =          {2025},
}

@article{KTLift1,
  author =        {Kawakami, Tatsuro and Tanaka, Hiromu},
  journal =       {preprint available at arXiv:2503.10236v1},
  title =         {Liftability and vanishing theorems for {F}ano threefolds in positive characteristic {I}},
  year =          {2025},
}

@article {CS24,
    AUTHOR = {\v{C}esnavi\v{c}ius, K\k{e}stutiss and Scholze, Peter},
     TITLE = {Purity for flat cohomology},
   JOURNAL = {Ann. of Math. (2)},
  FJOURNAL = {Annals of Mathematics. Second Series},
    VOLUME = {199},
      YEAR = {2024},
    NUMBER = {1},
     PAGES = {51--180},
      ISSN = {0003-486X,1939-8980},
   MRCLASS = {14F20 (14F22 14F30 14H20 18G90)},
  MRNUMBER = {4681144},
MRREVIEWER = {Shizhang\ Li},
       DOI = {10.4007/annals.2024.199.1.2},
       URL = {https://doi-org.kyoto-u.idm.oclc.org/10.4007/annals.2024.199.1.2},
}

@book {Kol13,
    AUTHOR = {Koll\'{a}r, J\'{a}nos},
     TITLE = {Singularities of the minimal model program},
    SERIES = {Cambridge Tracts in Mathematics},
    VOLUME = {200},
      NOTE = {With a collaboration of S\'{a}ndor Kov\'{a}cs},
 PUBLISHER = {Cambridge University Press, Cambridge},
      YEAR = {2013},
     PAGES = {x+370},
      ISBN = {978-1-107-03534-8},
   MRCLASS = {14E30 (14B05)},
  MRNUMBER = {3057950},
MRREVIEWER = {Tommaso De Fernex},
       DOI = {10.1017/CBO9781139547895},
       URL = {https://doi-org.utokyo.idm.oclc.org/10.1017/CBO9781139547895},
}

@book{Lam01,
  author    = {Lam, T. Y.},
  title     = {A First Course in Noncommutative Rings},
  series    = {Graduate Texts in Mathematics},
  volume    = {131},
  edition   = {Second},
  publisher = {Springer},
  year      = {2001}
}

@article {Laz86,
    AUTHOR = {Lazarsfeld, Robert},
     TITLE = {Brill-{N}oether-{P}etri without degenerations},
   JOURNAL = {J. Differential Geom.},
  FJOURNAL = {Journal of Differential Geometry},
    VOLUME = {23},
      YEAR = {1986},
    NUMBER = {3},
     PAGES = {299--307},
      ISSN = {0022-040X,1945-743X},
   MRCLASS = {14H10 (14J28)},
  MRNUMBER = {852158},
MRREVIEWER = {Ziv\ Ran},
       URL = {http://projecteuclid.org.kyoto-u.idm.oclc.org/euclid.jdg/1214440116},
}

@article {MS21,
    AUTHOR = {Ma, Linquan and Schwede, Karl},
     TITLE = {Singularities in mixed characteristic via perfectoid big
              {C}ohen-{M}acaulay algebras},
   JOURNAL = {Duke Math. J.},
  FJOURNAL = {Duke Mathematical Journal},
    VOLUME = {170},
      YEAR = {2021},
    NUMBER = {13},
     PAGES = {2815--2890},
      ISSN = {0012-7094,1547-7398},
   MRCLASS = {14G45 (13A35 14B05 14D10 14F18)},
  MRNUMBER = {4312190},
MRREVIEWER = {Ana\ Bravo},
       DOI = {10.1215/00127094-2020-0082},
       URL = {https://doi-org.kyoto-u.idm.oclc.org/10.1215/00127094-2020-0082},
}

@article {MSTWW,
    AUTHOR = {Ma, Linquan and Schwede, Karl and Tucker, Kevin and Waldron,
              Joe and Witaszek, Jakub},
     TITLE = {An analogue of adjoint ideals and {PLT} singularities in mixed
              characteristic},
   JOURNAL = {J. Algebraic Geom.},
  FJOURNAL = {Journal of Algebraic Geometry},
    VOLUME = {31},
      YEAR = {2022},
    NUMBER = {3},
     PAGES = {497--559},
      ISSN = {1056-3911,1534-7486},
   MRCLASS = {14G45 (13A35)},
  MRNUMBER = {4484548},
MRREVIEWER = {Sotiris\ Karanikolopoulos},
       DOI = {10.1090/jag/797},
       URL = {https://doi-org.kyoto-u.idm.oclc.org/10.1090/jag/797},
}

@article {MM81,
    AUTHOR = {Mori, Shigefumi and Mukai, Shigeru},
     TITLE = {Classification of {F}ano {$3$}-folds with {$B_{2}\geq 2$}},
   JOURNAL = {Manuscripta Math.},
  FJOURNAL = {Manuscripta Mathematica},
    VOLUME = {36},
      YEAR = {1981/82},
    NUMBER = {2},
     PAGES = {147--162},
      ISSN = {0025-2611},
   MRCLASS = {14J30 (14J10)},
  MRNUMBER = {641971},
MRREVIEWER = {Mary Schaps},
       DOI = {10.1007/BF01170131},
       URL = {https://doi.org/10.1007/BF01170131},
}

@incollection {MM83,
    AUTHOR = {Mori, Shigefumi and Mukai, Shigeru},
     TITLE = {On {F}ano {$3$}-folds with {$B_{2}\geq 2$}},
 BOOKTITLE = {Algebraic varieties and analytic varieties ({T}okyo, 1981)},
    SERIES = {Adv. Stud. Pure Math.},
    VOLUME = {1},
     PAGES = {101--129},
 PUBLISHER = {North-Holland, Amsterdam},
      YEAR = {1983},
   MRCLASS = {14J30},
  MRNUMBER = {715648},
MRREVIEWER = {I. Dolgachev},
       DOI = {10.2969/aspm/00110101},
       URL = {https://doi.org/10.2969/aspm/00110101},
}

@article {MM03,
    AUTHOR = {Mori, Shigefumi and Mukai, Shigeru},
     TITLE = {Erratum: ``{C}lassification of {F}ano 3-folds with {$B_2\geq
              2$}'' [{M}anuscripta {M}ath. {\bf 36} (1981/82), no. 2,
              147--162; {MR}0641971 (83f:14032)]},
   JOURNAL = {Manuscripta Math.},
  FJOURNAL = {Manuscripta Mathematica},
    VOLUME = {110},
      YEAR = {2003},
    NUMBER = {3},
     PAGES = {407},
      ISSN = {0025-2611},
   MRCLASS = {14J45 (14E30 14J30)},
  MRNUMBER = {1969009},
       DOI = {10.1007/s00229-002-0336-2},
       URL = {https://doi.org/10.1007/s00229-002-0336-2},
}

@article {Muk89,
    AUTHOR = {Mukai, Shigeru},
     TITLE = {Biregular classification of {F}ano {$3$}-folds and {F}ano
              manifolds of coindex {$3$}},
   JOURNAL = {Proc. Nat. Acad. Sci. U.S.A.},
  FJOURNAL = {Proceedings of the National Academy of Sciences of the United
              States of America},
    VOLUME = {86},
      YEAR = {1989},
    NUMBER = {9},
     PAGES = {3000--3002},
      ISSN = {0027-8424},
   MRCLASS = {14J30 (14J35 14J40)},
  MRNUMBER = {995400},
MRREVIEWER = {A. S. Tikhomirov},
       DOI = {10.1073/pnas.86.9.3000},
       URL = {https://doi-org.utokyo.idm.oclc.org/10.1073/pnas.86.9.3000},
}

@article {Muk10,
    AUTHOR = {Mukai, Shigeru},
     TITLE = {Curves and symmetric spaces, {II}},
   JOURNAL = {Ann. of Math. (2)},
  FJOURNAL = {Annals of Mathematics. Second Series},
    VOLUME = {172},
      YEAR = {2010},
    NUMBER = {3},
     PAGES = {1539--1558},
      ISSN = {0003-486X,1939-8980},
   MRCLASS = {14H45 (14C20 14H51 14M15)},
  MRNUMBER = {2726093},
MRREVIEWER = {Raquel\ Mallavibarrena},
       DOI = {10.4007/annals.2010.172.1539},
       URL = {https://doi-org.kyoto-u.idm.oclc.org/10.4007/annals.2010.172.1539},
}

@article{Sha,
  author =        {Shatova, Irina},
  journal =       {arXiv:2601.14709v1},
  title =         {Brill--Noether generality of curves and K3 surfaces},
  year =          {2026},
}

@article {Sho79a,
    AUTHOR = {\v{S}okurov, V. V.},
     TITLE = {The existence of a line on {F}ano varieties},
   JOURNAL = {Izv. Akad. Nauk SSSR Ser. Mat.},
  FJOURNAL = {Izvestiya Akademii Nauk SSSR. Seriya Matematicheskaya},
    VOLUME = {43},
      YEAR = {1979},
    NUMBER = {4},
     PAGES = {922--964, 968},
      ISSN = {0373-2436},
   MRCLASS = {14J30 (14M20)},
  MRNUMBER = {548510},
MRREVIEWER = {Miles Reid},
}

@article {Sho79b,
    AUTHOR = {\v{S}okurov, V. V.},
     TITLE = {Smoothness of a general anticanonical divisor on a {F}ano
              variety},
   JOURNAL = {Izv. Akad. Nauk SSSR Ser. Mat.},
  FJOURNAL = {Izvestiya Akademii Nauk SSSR. Seriya Matematicheskaya},
    VOLUME = {43},
      YEAR = {1979},
    NUMBER = {2},
     PAGES = {430--441},
      ISSN = {0373-2436},
   MRCLASS = {14J30},
  MRNUMBER = {534602},
MRREVIEWER = {Werner Kleinert},
}

@misc{SP,
  author =        {{The} {Stacks Project Authors}},
  howpublished =  {\url{http://stacks.math.columbia.edu}},
  title =         {\itshape {S}tacks {P}roject},
}

@article {TY23,
    AUTHOR = {Takamatsu, Teppei and Yoshikawa, Shou},
     TITLE = {Minimal model program for semi-stable threefolds in mixed
              characteristic},
   JOURNAL = {J. Algebraic Geom.},
  FJOURNAL = {Journal of Algebraic Geometry},
    VOLUME = {32},
      YEAR = {2023},
    NUMBER = {3},
     PAGES = {429--476},
      ISSN = {1056-3911,1534-7486},
   MRCLASS = {14E30 (14G45)},
  MRNUMBER = {4622257},
MRREVIEWER = {Sotiris\ Karanikolopoulos},
       DOI = {10.1090/jag/813},
       URL = {https://doi-org.kyoto-u.idm.oclc.org/10.1090/jag/813},
}

@article {Tak89,
    AUTHOR = {Takeuchi, Kiyohiko},
     TITLE = {Some birational maps of {F}ano {$3$}-folds},
   JOURNAL = {Compositio Math.},
  FJOURNAL = {Compositio Mathematica},
    VOLUME = {71},
      YEAR = {1989},
    NUMBER = {3},
     PAGES = {265--283},
      ISSN = {0010-437X},
   MRCLASS = {14J30 (14E05)},
  MRNUMBER = {1022045},
MRREVIEWER = {Peter Nielsen},
       URL = {http://www.numdam.org/item?id=CM_1989__71_3_265_0},
}

@article {Tan18b,
    AUTHOR = {Tanaka, Hiromu},
     TITLE = {Behavior of canonical divisors under purely inseparable base
              changes},
   JOURNAL = {J. Reine Angew. Math.},
  FJOURNAL = {Journal f\"{u}r die Reine und Angewandte Mathematik. [Crelle's
              Journal]},
    VOLUME = {744},
      YEAR = {2018},
     PAGES = {237--264},
      ISSN = {0075-4102},
   MRCLASS = {14E30},
  MRNUMBER = {3871445},
MRREVIEWER = {Sung Rak Choi},
       DOI = {10.1515/crelle-2015-0111},
       URL = {https://doi.org/10.1515/crelle-2015-0111},
}

@article {Tan24,
    AUTHOR = {Tanaka, Hiromu},
     TITLE = {Bertini theorems admitting base changes},
   JOURNAL = {J. Algebra},
  FJOURNAL = {Journal of Algebra},
    VOLUME = {644},
      YEAR = {2024},
     PAGES = {64--125},
      ISSN = {0021-8693,1090-266X},
   MRCLASS = {14G17 (14D06)},
  MRNUMBER = {4695615},
MRREVIEWER = {Lei\ Zhang},
       DOI = {10.1016/j.jalgebra.2023.12.038},
       URL = {https://doi-org.kyoto-u.idm.oclc.org/10.1016/j.jalgebra.2023.12.038},
}

@article{FanoI,
  author =        {Tanaka, Hiromu},
  journal =       {arXiv:2308.08121},
  title =         {Fano threefolds in positive characteristic {I}},
  year =          {2023},
}

@article{FanoII,
  author =        {Tanaka, Hiromu},
  journal =       {arXiv:2308.08122},
  title =         {Fano threefolds in positive characteristic {II}},
  year =          {2023},
}

@article{FanoIII,
  author =        {Asai, Masaya and Tanaka, Hiromu},
  journal =       {arXiv:2308.08124},
  title =         {Fano threefolds in positive characteristic {III}},
  year =          {2023},
}

@article{FanoIV,
  author =        {Tanaka, Hiromu},
  journal =       {arXiv:2308.08127},
  title =         {Fano threefolds in positive characteristic {IV}},
  year =          {2023},
}

@article {Yob19,
    AUTHOR = {Yobuko, Fuetaro},
     TITLE = {Quasi-{F}robenius splitting and lifting of {C}alabi-{Y}au
              varieties in characteristic {$p$}},
   JOURNAL = {Math. Z.},
  FJOURNAL = {Mathematische Zeitschrift},
    VOLUME = {292},
      YEAR = {2019},
    NUMBER = {1-2},
     PAGES = {307--316},
      ISSN = {0025-5874,1432-1823},
   MRCLASS = {14J32 (13F35 14G17)},
  MRNUMBER = {3968903},
MRREVIEWER = {Tyler\ L.\ Kelly},
       DOI = {10.1007/s00209-018-2198-7},
       URL = {https://doi-org.kyoto-u.idm.oclc.org/10.1007/s00209-018-2198-7},
}

\end{document}